\documentclass[11pt,a4paper]{article}

\usepackage{titlesec}
\usepackage{fancyhdr}
\usepackage{a4wide}
\usepackage{graphicx}
\usepackage{float}
\usepackage{amssymb}
\usepackage{amsmath}
\usepackage{amsfonts}
\usepackage{amsthm}
\usepackage{color}
\usepackage{mathrsfs}
\usepackage{array}
\usepackage{eucal}
\usepackage{tikz}
\usepackage[T1]{fontenc}
\usepackage{inputenc}
\usepackage[english]{babel}
\usepackage{lmodern}
\usepackage{hyperref}
\usepackage{geometry}
\usepackage{changepage}
\usepackage{bm}
\changepage{0pt}{}{}{}{}{0pt}{}{0pt}{6pt}
\usepackage[numbers]{natbib}
\usepackage{hyperref}
\hypersetup{
pdfpagemode=none,
pdftoolbar=true,        
pdfmenubar=true,        
pdffitwindow=false,     
pdfstartview={Fit},    
pdftitle={A continuation method for building invisible obstacles in waveguides},    
pdfauthor={A. Bera, A.-S. Bonnet-BenDhia, L. Chesnel},     
pdfsubject={},  
pdfcreator={A. Bera, A.-S. Bonnet-BenDhia, L. Chesnel},   
pdfproducer={A. Bera, A.-S. Bonnet-BenDhia, L. Chesnel}, 
pdfkeywords={}, 
pdfnewwindow=true,      
colorlinks=true,       
linkcolor=magenta,          
citecolor=red,        
filecolor=cyan,      
urlcolor=blue           
}

\newcommand{\dsp}{\displaystyle}

\newcommand{\eps}{\varepsilon}
\newcommand{\om}{\omega}
\newcommand{\Om}{\Omega}
\newcommand{\mrm}[1]{\mathrm{#1}}

\newcommand{\Cplx}{\mathbb{C}}
\newcommand{\N}{\mathbb{N}}
\newcommand{\R}{\mathbb{R}}

\newcommand{\mL}{\mrm{L}}
\newcommand{\mH}{\mrm{H}}

\newcommand{\mX}{\mrm{X}}

\newcommand{\loc}{\mbox{\scriptsize loc}}

\renewcommand{\ker}{\mrm{ker}}

\newcommand{\G}{G}
\newcommand{\K}{K}
\newcommand{\radius}{r}

\newtheorem{theorem}{Theorem}[section]
\newtheorem{lemma}[theorem]{Lemma}
\newtheorem{remark}[theorem]{Remark}

\newtheorem{proposition}[theorem]{Proposition}

\begin{document}

~\vspace{0.0cm}
\begin{center}
{\sc \bf\huge A continuation method for building \\[6pt] invisible obstacles in waveguides}
\end{center}

\begin{center}
\textsc{Antoine Bera}$^1$, \textsc{Anne-Sophie Bonnet-BenDhia}$^1$, \textsc{Lucas Chesnel}$^2$\\[16pt]
\begin{minipage}{0.96\textwidth}
{\small
$^1$ Laboratoire  Poems,  CNRS/INRIA/ENSTA Paris, Institut  Polytechnique de Paris, 828 Boulevard des Mar\'echaux, 91762 Palaiseau, France;\\
$^2$ INRIA/Centre de math\'ematiques appliqu\'ees, \'Ecole Polytechnique,  Institut  Polytechnique de Paris, Route de Saclay, 91128 Palaiseau, France.\\[10pt]
E-mails: \texttt{anne-sophie.bonnet-bendhia@ensta-paris.fr}, \texttt{lucas.chesnel@inria.fr}\\[-14pt]
\begin{center}
(\today)
\end{center}
}
\end{minipage}
\end{center}
\vspace{0.4cm}

\noindent\textbf{Abstract.} We consider the propagation of acoustic waves at a given wavenumber in a waveguide which is unbounded in one direction.  We explain how to construct penetrable obstacles characterized by a physical coefficient $\rho$ which are invisible in various ways. In particular, we focus our attention on invisibility in reflection (the reflection matrix is zero), invisibility in reflection and transmission (the scattering matrix is the same as if there were no obstacle) and relative invisibility (two different obstacles have the same scattering matrix). To study these problems, we use a continuation method which requires to compute the scattering matrix $\mathbb{S}(\rho)$ as well as its differential with respect to the material index $d\mathbb{S}(\rho)$. The justification of the method also needs for the proof of abstract results of ontoness of well-chosen functionals constructed from the terms of $d\mathbb{S}(\rho)$. We provide a complete proof of the results in monomode regime when the wavenumber is such that only one mode can propagate. And we give all the ingredients to implement the method in multimode regime. We end the article by presenting numerical results to illustrate the analysis.\\


\section{Introduction}\label{Introduction}

We consider an acoustic waveguide $\Om$ which is unbounded in one direction. We assume that it contains a bounded penetrable obstacle characterized by a physical coefficient $\rho_0$ and that the propagation of waves is governed by the Helmholtz equation $\Delta u + k^2(1+\rho_0) u=0$ together with homogeneous Neumann boundary conditions. We work at a given wavenumber $k>0$ so that a finite number of modes can propagate in $\Om$. In general, the presence of the obstacle perturbs the propagation of modes resulting in reflection phenomena on one side of the obstacle and conversion phenomena on the other side. We denote by $\mathbb{S}(\rho_0)$ the corresponding scattering matrix whose entries are the reflection and transmission coefficients on these propagating modes. In absence of obstacle, that is when $\rho_0\equiv0$, waves propagate through the structure without being scattered. The initial motivation of this article is to construct invisible obstacles, that is to find $\rho\not\equiv0$ such that $\mathbb{S}(\rho)=\mathbb{S}(0)$. In this case, up to some remainders which are exponentially decaying at infinity, the fields are the same in the waveguide with and without obstacle. A less ambitious objective is to construct obstacles which are simply non reflecting. In this case, we want to find $\rho\not\equiv0$ such that the reflection coefficients are all zero.\\
\newline
More generally, let us extract from $\mathbb{S}(\rho_0)$ a certain number of coefficients that we want to control and let us gather them in the vector $\G(\rho_0)$. The goal of this article is to explain how to find $\rho\not\equiv\rho_0$, where $\rho_0$ is given, such that $\G(\rho)=\G(\rho_0)$. The difficulty in this work lies in the fact that the dependence of the scattering coefficients with respect to $\rho$ is not simple, in particular it is not linear. To solve  our problem, as proposed in \cite{na582} to deal with the problem of invisibility, we shall use a continuation method. Let us describe the methodology.\\
\newline
First, we select some elements from $\G$ and we gather them in a vector $F$ such that  
 the relation $F(\rho)=F(\rho_0)$ guarantees that the identity $\G(\rho)=\G(\rho_0)$ holds. One could be tempted to gather in $F$ all the elements of $\G$. However in general this is not so simple because in order the method below to work, we need $dF(\rho_0)$, the differential of $F$ at $\rho_0$, to be onto in some spaces to define. But $F$ is constructed from $\G$, and so from $\mathbb{S}$ which is unitary and symmetric. Therefore, there is a strong structure for the scattering coefficients and one has to choose carefully the elements of $F$ to avoid to have redundant information. We emphasize that deciding which $F$ to consider in general is not straightforward. The second step in the procedure consists in adapting the proof of the implicit function theorem by looking for $\rho$ such that $F(\rho)=F(\rho_0)$ and such that $\rho$ is a small perturbation of $\rho_0$. More precisely, let us look for $\rho$ as $\rho=\rho_0+\eps\mu$ where $\eps>0$ is small and where $\mu$ has to be determined. Then a Taylor expansion gives 
\begin{equation}\label{EqnPtFixe}
F(\rho_0+\eps\mu)=F(\rho_0)+\eps\,dF(\rho_0)(\mu)+\eps^2\widetilde{F}^{\eps}(\mu),
\end{equation}
where $\widetilde{F}^{\eps}(\mu)$ is an abstract remainder which depends non-linearly on $\eps$, $\mu$. Finally, the last step consists in finding a non zero $\mu^{\mrm{sol}}$ in some appropriate set of functions, such that $dF(\rho_0)(\mu^{\mrm{sol}})=-\eps\,\widetilde{F}^{\eps}(\mu^{\mrm{sol}})$ by solving a fixed point equation. Then from (\ref{EqnPtFixe}), we see that this implies $F(\rho_0+\eps\mu^{\mrm{sol}})=F(\rho_0)$. We emphasize that a priori $\eps$ has to be small to guarantee that the operator appearing in the right hand side of the fixed point equation is a contraction mapping. For this reason, what we construct are invisible perturbations which are, a priori, of small amplitude. Note that since $dF(\rho_0)$ is surely not injective (it is a linear map between an infinite dimensional space and a finite dimensional one), there are in fact infinitely many solutions of $F(\rho_0+\eps\mu^{\mrm{sol}})=F(\rho_0)$. This confers some flexibility to the algorithm, which can be exploited to impose some constraints of feasibility to the obstacle that is built.\\ 
\newline
This idea was introduced in \cite{BoNa13} (see also \cite{BLMN15} for numerical examples) to construct non reflecting perturbations of the wall, instead of a penetrable obstacle, in monomode regime. Note that in monomode regime, the scattering matrix is of the form 
\[
\mathbb{S}=\left(\begin{array}{cc}
R^+ & T\\
T & R^-
\end{array}
\right)\in\Cplx^{2\times2},
\]
where $R^{\pm}$ are reflection coefficients and where $T$ is a transmission coefficient. From conservation of energy, we have $|R^{\pm}|+|T|^2=1$. Non reflecting obstacles are such that $R^{\pm}=0$, and therefore such that $|T|=1$. If we do not impose $T=1$, there is a possible phase shift between the incident and transmitted fields. In \cite{ChNa16}, it is shown how this phase shift can be removed by working with singular perturbations of the walls instead of smooth ones, achieving $T=1$ (invisibility in reflection and transmission).  For the construction of families of small obstacles which are collectively non reflecting, we refer the reader \cite{na648}; for an application to water-waves, see \cite{na582}; for the construction of invisible penetrable obstacles in free space and for a problem appearing in medical imaging, see respectively \cite{BoCN15} and \cite{ChHS15}.\\
\newline
In the present work, we extend the above mentioned works in the following directions. First, we provide results in multimode regime. Second, we explain how to reiterate the process. More precisely, we said above that what we obtain are small invisible perturbations. But, once an invisible obstacle has been constructed, a natural idea is to use it as a starting point to get larger invisible defects. Note that preliminary results to address this problem can be found in \cite{na582}. The implementation of this procedure requires to solve various questions. In particular, we have to compute the differential of $F$ at a point $\rho_0$. Then, and this is the most difficult point, we have to prove that $dF(\rho_0)$ is onto. When $\rho_0\equiv0$, we get explicit formulas and the ontoness of $dF(0)$ can be established quite directly, at least in monomode regime. On the other hand, when $\rho_0\not\equiv0$, the expression of $dF(\rho_0)$ involves abstract functions and the ontoness of $dF(\rho_0)$ is not obvious. Clarifying completely this question in monomode regime is the main outcome of this article.\\ 
\newline 
In the present work, we also study a question of relative invisibility. From a general point of view, for any $\rho_0$ given, an interesting objective is to construct $\rho\not\equiv\rho_0$ such that $\mathbb{S}(\rho)=\mathbb{S}(\rho_0)$. This means that the two obstacles, with coefficients $\rho$ and $\rho_0$, are indistinguishable by using standard scattering measurements. To do that, first we have to understand how to define $F$ as explained above so that the relation $F(\rho)=F(\rho_0)$ implies $\mathbb{S}(\rho)=\mathbb{S}(\rho_0)$. This will oblige us to understand finely the consequences of the structure of $\mathbb{S}$ on the properties of its differential. We will see that the choice of $F$ depends on the value of $\mathbb{S}(\rho_0)$.\\
\newline
Questions of invisibility in waveguides are studied in particular in the context of Perfect Transmission Resonances (PTRs), see e.g. \cite{Shao94,PoGP99,LeKi01,Zhuk10,MrMK11}. For mathematical approaches different from the continuation method, we refer the reader to \cite{ChNPSu,ChPaSu,ChPa19} and \cite{ShVe05,ShTu12,ShWe13,AOMFM14,AbSh16,ChNa18}. All these studies work only in monomode regime and rely on symmetry properties of the geometry. In multimode regime, one can solve spectral problems with ingoing/outgoing conditions at infinity \cite{BoCP18,SwHS19}.  In order to attack these problems of invisibility, one can also use techniques of optimization (see e.g. \cite{LDOHG19,LGHDO19}). However, the functionals which are involved in the process are not convex and local minima exist. In other words, they only provide approximated invisible obstacles. The continuation method we consider is a bit more restrictive but offers the advantage of providing exact solutions. For other literature concerning invisibility, we can also read the article \cite{PaMB17}.\\
\newline 
The article is organized as follows. First, we introduce the setting, present the different problems that we will consider and describe the mechanism of the continuation method. Then in Section \ref{SectionS}, we give an explicit form for the scattering matrix and we compute its differential with respect to the material index. Then we implement the continuation method in monomode regime in Section \ref{SectionMono}. More precisely, we show results of ontoness for the differentials of some well-chosen functionals constructed from the elements of the scattering matrix. 
In Section \ref{SectionNumerics}, we give numerical illustrations of the results. Then we explain how to impose additional constraints on the invisible obstacles we construct in Section \ref{SectionConstraints}. Finally, we end with some concluding remarks and open questions.

\section{Setting}\label{setting}

\subsection{The scattering problem}

\begin{figure}[!ht]
\centering
\begin{tikzpicture}[scale=0.92]
\draw[fill=gray!16,draw=none](-3,0) rectangle (3,2);
\draw (-3,0)--(3,0);
\draw (-3,2)--(3,2);

\draw [dashed](-3,0)--(-4,0);
\draw [dashed](3,0)--(4,0);
\draw [dashed](-3,2)--(-4,2);
\draw [dashed](3,2)--(4,2);
\node at (-2.7,0.2){\small$\Omega$};
\draw [fill=gray!90,draw=none] plot [smooth cycle, tension=1] coordinates {(-0.6,0.9) (0,0.5) (0.7,1) (0.5,1.5) (-0.2,1.4)};
\node at (0,1){\small$\rho\ne 0$};
\node at (2.3,1){\small$\rho= 0$};
\draw (-1.7,-0.1)--(-1.7,0.1) ;
\draw (1.7,-0.1)--(1.7,0.1);
\node at (-1.7,-0.4){\small$-\ell$};
\node at (1.7,-0.4){\small$+\ell$};
\end{tikzpicture}
\caption{Setting. \label{FigureSetting}} 
\end{figure}

\noindent In this work, we are interested in the propagation of acoustic waves in time-harmonic regime in the waveguide $\Om = \{z=(x,y)\in\R\times(0;1)\}$ in presence of a penetrable obstacle. This leads us to consider the equations
\begin{equation}\label{InitialPb}
\begin{array}{|rl}
\Delta u + k^2(1+\rho) u = 0 & \mbox{ in }\Om \\[3pt]
\partial_{y} u=0 & \mbox{ on }\partial\Om.
\end{array}
\end{equation}
In (\ref{InitialPb}), $u$ is the complex valued acoustic pressure and $k:=\om/c$ denotes the wavenumber, $\om$ being the pulsation and $c$ the sound speed. Moreover, $\Delta$ stands for the 2D Laplace operator. Finally, $\rho\in\mL^{\infty}(\Om)$ is a real-valued physical coefficient which characterizes the obstacle. It is such that $\rho=0$ in $\Om\setminus\mathscr{O}$ where here and up to the end of the document, $\mathscr{O}$ is a  given non-empty bounded open set. In what follows, we will often identify the obstacle to the physical coefficient $\rho$. Additionally, with a slight abuse, we shall make no distinction between the elements of $\mL^{\infty}(\mathscr{O})$ and their extensions by zero to $\Om$ writing for example that $\rho\in\mL^{\infty}(\mathscr{O})$. We introduce a parameter $\ell>0$ such that $\rho(x,y)=0$ for $|x|\ge \ell$ (see Figure \ref{FigureSetting}). Using separation of variables, one can compute the solutions of the unperturbed problem 
\begin{equation}\label{UnperturbedPb}
\begin{array}{|rl}
\Delta u + k^2 u = 0 & \mbox{ in }\Om \\[3pt]
\partial_{y} u=0 & \mbox{ on }\partial\Om
\end{array}
\end{equation}
which are called the \textit{modes} of the waveguide. More precisely, setting 
  for $n\in\N$,
\begin{equation}\label{defCos}
\varphi_n(y)=\alpha_n\cos(n\pi y)\quad\mbox{ with }\quad\begin{array}{|lcl}
\alpha_0&=&1 \\
\alpha_n&=&\sqrt{2}\quad\mbox{for }n\ge1,
\end{array}
\end{equation}
the modes are  defined as follows:
\begin{equation}\label{defModes}
w^{\pm}_n(x,y)= (2|\beta_n|)^{-1/2} e^{\pm i \beta_n x}\varphi_n(y)\quad \mbox{ with }\ \beta_n:=\sqrt{k^2-n^2\pi^2}.
\end{equation}
Here the complex square root is chosen so that if $\gamma=r e^{i\eta}$ with $r\ge 0$ and $ \eta \in[0;2\pi)$, then $\sqrt{\gamma}=\sqrt{r}e^{i\gamma/2}$. In (\ref{defModes}), the normalization coefficients  are chosen so that the scattering matrix defined in (\ref{defScatteringMatrix}) is unitary. In the following, we shall assume that the wavenumber $k$ satisfies $(N-1)\pi<k<N\pi$ for some $N\in\N$. Then, according to the value of $n$, the modes $w^{\pm}_n$ adopt different behaviours. For $n\ge N$ we have $\beta_n=i \sqrt{n^2\pi^2-k^2}$ and the function $w^{+}_n$ (resp. $w^{-}_n$) decays exponentially at $+\infty$ (resp. $-\infty$) while it grows exponentially at $-\infty$ (resp. $+\infty$). For $n\in\{0,\dots,N-1\}$, with a convention in time in $e^{-i\om t}$, the function $w^{+}_n$ (resp. $w^{-}_n$) corresponds to a right-going (resp. left-going) wave. The cut-off values $k=n\pi$ for the wavenumber will not be considered in this work.\\
\newline
Let us consider now the perturbed situation (\ref{InitialPb}). Suppose that for some $n$, $0\le n \le N-1$, the wave $w^+_{n}$ (resp. $w^-_{n}$) travels from $-\infty$ (resp. $+\infty$), in the positive (resp. negative) direction of the $(Ox)$ axis and is scattered by the obstacle. Then the total field $u$ satisfies the equations (\ref{InitialPb}), which must be supplemented with radiation conditions at $\pm\infty$. We will say that a function $v\in\mrm{H}^1_{\loc}(\Om)$ which satisfies  (\ref{InitialPb}) is \textit{outgoing} if it admits the decompositions
\begin{equation}\label{scatteredField}
v =  \sum_{n=0}^{N-1} s^{-}_{n}\,w^-_n+\tilde{v}\quad\mbox{ for }x\le -\ell\qquad\mbox{ and }\qquad v =  \sum_{n=0}^{N-1}s^{+}_{n}\,w^+_n+ \tilde{v}\quad\mbox{ for }x\ge +\ell,
\end{equation}
for some constants $s^{\pm}_n\in\Cplx$ and some $\tilde{v}\in\mH^1(\Om)$. Using Fourier decomposition, one can verify that the remainder $\tilde{v}$ in (\ref{scatteredField}) decomposes only on the exponentially decaying modes at infinity so that $\tilde{v}$ is exponentially decaying. Now let $u_i$ be a combination of the propagating modes $w_n^{\pm}$, $n=0,\dots,N-1$. The scattering problem we consider states 
\begin{equation}\label{PbChampTotalBIS}
\begin{array}{|rl}
\multicolumn{2}{|l}{\mbox{Find }u\in\mrm{H}^1_{\loc}(\Om) \mbox{ such that }u-u_i\mbox{ is outgoing and } }\\[3pt]
\Delta u + k^2(1+\rho) u = 0 & \mbox{ in }\Om \\[3pt]
\partial_{y} u=0 & \mbox{ on }\partial\Om.
\end{array}
\end{equation}
It is known (see e.g. \cite[Chap. 5, \S3.3, Thm. 3.5 p. 160]{NaPl94}) that problem (\ref{PbChampTotalBIS}) always admits a solution. Moreover uniqueness holds if and only if so-called \textit{trapped modes} do not exist. We remind the reader that $u\not\equiv 0$ is said to be a trapped mode if it solves the homogeneous problem (\ref{InitialPb}) and is of finite energy (belongs to $\mL^2(\Om)$). Using again Fourier decomposition, one can check that trapped modes are exponentially decaying at infinity. If there exists a family of linearly independent trapped modes $\psi_1,\dots,\psi_p$ for some $p\ge1$, we shall call $u$ the solution to (\ref{PbChampTotalBIS}) which is orthogonal to $\psi_1,\dots,\psi_p$ for the inner product of $\mrm{L}^2(\Om)$. We shall use this convention all over the document. In the following,  $u-u_i$ (resp. $u$) will be referred to as the \textit{scattered} (resp. \textit{total}) field associated with the \textit{incident} field $u_i$. For $u_i=w^{\pm}_m$, $m=0,\dots,N-1$, we shall denote $u_m^{\pm}$ the corresponding total field. From (\ref{scatteredField}), we know that $u_m^{+}$, $u_m^{-}$ decompose as 
\begin{equation}\label{scatteredFieldParticular}
\begin{array}{ll}
 u^+_m=w^+_m+\dsp\sum_{n=0}^{N-1}R^+_{mn}w^-_n+\tilde{u}_m^{+}\mbox{ for }x\le-\ell,&\qquad u^+_m=\dsp\sum_{n=0}^{N-1}T^+_{mn} w^+_n+\tilde{u}_m^{+}\mbox{ for }x\ge \ell, \\[12pt]
u^-_m=\dsp\sum_{n=0}^{N-1}T^-_{mn} w^-_n+\tilde{u}_m^{-}\mbox{ for }x\le-\ell,&\qquad u^-_m=w^-_m+\dsp\sum_{n=0}^{N-1}R^-_{mn}w^+_n+\tilde{u}_m^{-}\mbox{ for }x\ge \ell.
\end{array}
\end{equation}
The coefficients $R^{\pm}_{mn}\in\Cplx$ (resp. $T^{\pm}_{mn}\in\Cplx$) are called \textit{reflection} (resp. \textit{transmission}) coefficients. They form the \textit{scattering matrix} defined by  
\begin{equation}\label{defScatteringMatrix}
\mathbb{S}=\left(\begin{array}{cc}
\mathcal{R}^+ & \mathcal{T}^+\\
\mathcal{T}^- & \mathcal{R}^-
\end{array}
\right)\in\Cplx^{2N\times2N}\quad\mbox{ with }\quad\begin{array}{|l}
\mathcal{R}^{\pm}=(R^{\pm}_{mn})_{0\le m,n\le N-1}\\[6pt]
\mathcal{T}^{\pm}=(T^{\pm}_{mn})_{0\le m,n\le N-1}.
\end{array}
\end{equation}
Note that conservation of energy allows one to show that any outgoing function solving (\ref{InitialPb}) is exponentially decaying at infinity and therefore is a trapped mode. As a consequence, by linearity, two solutions of (\ref{PbChampTotalBIS}) for the same $u_i$ have the same scattering coefficients. It is known that $\mathbb{S}$ is symmetric ($\mathbb{S}^{\top}=\mathbb{S}$) and unitary ($\mathbb{S}\overline{\mathbb{S}}^{\top}=\mrm{Id}^{2N\times 2N}$). For the sake of clarity, we remind the proof of these two facts in Proposition \ref{SymUnitary} below. Note that from time to time, as above, we omit to write the dependence of the scattering coefficients with respect to $\rho$ when there is no risk of confusion.
	
\subsection{A few examples of problems of practical interest}\label{paragraphProblems}

Now that we have defined the scattering matrix, we describe more precisely the problems that we wish to study in the following. The scattering matrix $\mathbb{S}=\mathbb{S}(\rho)$ is a non-linear function of $\rho$. Generally speaking, we will try to impose prescribed values for certain scattering coefficients by playing with the parameter $\rho$ in (\ref{InitialPb}). Let us gather the real and imaginary parts of the coefficients of interest in some real valued vector $\G(\rho)$. We want to solve problems of the form
\begin{equation} \label{problem0}
\mbox{Find } \ \rho \in \mL^{\infty}(\mathscr{O}) \ \mbox{such that} \ \G(\rho)=\G(\rho_0),
\end{equation}
where $\rho_0\in\mL^{\infty}(\mathscr{O})$ is given. To proceed, we consider the problem 
\begin{equation} \label{problem1}
\mbox{Find } \ \rho \in \mL^{\infty}(\mathscr{O}) \ \mbox{such that} \ F(\rho)=F(\rho_0),
\end{equation}
where $F$ contains some elements of $\G$. As explained in the introduction, all the game consists in choosing carefully $F$ so that solving (\ref{problem1}) gives a solution of (\ref{problem0}). This depends on $\G$ and on $\rho_0$.\\
\newline
We will distinguish below three types of problems, related to different kinds of experiments. If some observer has only access to backscattering measurements, constructing an invisible obstacle amounts to cancel reflection coefficients. It is the \textit{invisibility in reflection}. If measurements are available on both sides of the waveguide, creating an invisible obstacle requires to prescribe values for both reflection and transmission coefficients. It is the \textit{invisibility in reflection and transmission}. Finally, if we do not want the obstacle to be invisible but instead to be indistinguishable from another one, we will speak of \textit{relative invisibility}.\\ 
\newline
\textbf{Invisibility in reflection.\quad}To achieve invisibility in reflection, we need to impose $\mathcal{R}^+(\rho)=0^{N\times N}$. In this case, we shall say that the obstacle is \textit{non reflecting}. In particular, an observer producing right-going waves and measuring the response of the system at $x=-L$ with $L>0$ a bit large will see nothing but a field which is exponentially decaying. Due to noise in measurements, this response will not be distinguishable from  that of the reference waveguide. In (\ref{problem0}), we set $\G(\rho)=\left( \Re e\,R_{mn}^+, \Im m\,R_{mn}^+ \right)_{0\le m,n \le N-1}$ and $\rho_0\equiv0$. Since the scattering matrix is symmetric, one does not need to impose both $R_{mn}^+=0$ and $R_{nm}^+=0$. As a consequence, in (\ref{problem1}) one can take
\begin{equation}\label{NoReflection}
	F(\rho)=\left( \Re e\,R_{mn}^+, \Im m\,R_{mn}^+ \right)_{0\le m \le n \le N-1} \in \R^{N(N+1)}.
\end{equation}
\textbf{Invisibility in reflection and transmission.\quad}One can desire to control the transmission of waves through the waveguide. To begin with, assume for example that we want to have complete transmission in energy only for one given incident mode $w_m^+$. In other words, we want to have no reflection ($R_{mn}^+=0$ for $n=0,\dots,N-1$) and no modal conversion ($T_{mn}^+=0$ for $n=0,\dots,N-1$ with $n\ne m$). Since $\mathbb{S}(\rho)$ is unitary, this is equivalent to impose $|T_{mm}^+|=1$. Therefore, it is tempting to set $F(\rho)= |T_{mm}^+|$. Unfortunately we will see that the continuation technique we use below fails with this choice of $F(\rho)$. Instead it is better to work with $\rho_0\equiv0$ and the a priori much more complicated functional
\begin{equation}\label{formulation_inv_partielle_mode_n}
	F(\rho)=\begin{array}{l}
	\left(  ( \Re e\,R_{mn}^+,\Im m\,R_{mn}^+)_{0\le n\le N-1},( \Re e\,T_{mn}^+,\Im m\,T_{mn}^+)_{0\le n\ne m \le N-1}\right)  \in \R^{4N-2}.
	\end{array}
\end{equation}
In the previous setting, the transmitted field may exhibit a shift of phase with respect to the incident mode. This is due to the fact that we impose $|T_{mm}^+| =1$ and not $T_{mm}^+=1$. In order to impose $T_{mm}^+=1$, since $\mathbb{S}(\rho)$ is unitary, one may take $\rho_0\equiv0$ and solve $F(\rho)= \Re e\,T_{mm}^+$=1. But again, we will see that our technique does not allow one to deal with this choice. Instead, it is better to work with $\rho_0\equiv0$ and 	 
\begin{equation}\label{formulation_inv_mode_n}
	F(\rho)=\left(  ( \Re e\,R_{mn}^+,\Im m\,R_{mn}^+)_{0\le n\le N-1},( \Re e\,T_{mn}^+,\Im m\,T_{mn}^+)_{0\le n\ne m \le N-1},\Im m\,T_{mm}^+	\right)  \in \R^{4N-1}.
\end{equation}
The only difference between  (\ref{formulation_inv_partielle_mode_n}) and  (\ref{formulation_inv_mode_n}) is that in (\ref{formulation_inv_mode_n}) we impose  additionally $ \Im m\,T_{mm}^+ = 0$.\\
\newline
Finally, to impose complete invisibility (complete transmission without phase shift) for all the incident modes $w_m^+$, $m=1,\dots,N$, we must have $\mathcal{R}^+(\rho)=0^{N\times N}$ and $\mathcal{T}^+(\rho)=\mrm{Id}^{N\times N}$. Observe that this is enough to guarantee that $\mathbb{S}(\rho)=\mrm{Id}^{2N\times2N}$ because $\mathbb{S}(\rho)$ is symmetric and unitary. Therefore, in this situation, we also have perfect invisibility for left-going incident waves. Let us see how to define $F$ in this case. First we impose $\mathcal{R}^+(\rho)=0$ working as in (\ref{NoReflection}). Then if we impose to the first line of $\mathcal{T}^+(\rho)$ to be equal to $(1,0,\dots,0)$, since $\mathbb{S}(\rho)$ is unitary, the first column of $\mathcal{T}^+(\rho)$ will be equal to $(1,0,\dots,0)^{\top}$. Iterating the process, we see that it is sufficient to cancel both the terms which are on the triangular upper part of $\mathcal{T}^+(\rho)$ and the imaginary part of the diagonal terms of $\mathcal{T}^+(\rho)$. As a consequence, we shall set $\rho_0\equiv0$ and 	 
\begin{equation}\label{formulation_inv_parfaite}
\left\{\hspace{-0.15cm}\begin{array}{l}
	F(\rho)=\left(  ( \Re e\,R_{mn}^+,\Im m\,R_{mn}^+)_{0\le m\le n\le N-1},( \Re e\,T_{mn}^+,\Im m\,T_{mn}^+)_{0\le m<n  \le N-1},\right.\\[2pt]
\left.\hspace{8.8cm}	(\Im m\,T_{mm}^+)_{0\le m \le N-1}	\right)  \in \R^{N(2N+1)}.
\end{array}\right.	
\end{equation}
\textbf{Relative invisibility.\quad}For a given $\rho_0\in\mL^{\infty}(\mathscr{O})$, one can be interested in finding $\rho\not\equiv\rho_0$ such that $\mathbb{S}(\rho)=\mathbb{S}(\rho_0)$. In other words, one can wish to find two different obstacles having the same scattering matrices. In this case, the choice of the good functional $F$ in (\ref{problem1}) depends on the value of $\mathbb{S}(\rho_0)$. We will discuss this case later.\\
\newline
We just had a glimpse	 of the variety of problems which write under the form (\ref{problem1}). In what follows, we will focus our attention on the problems of non reflecting (\ref{NoReflection}), perfectly invisible (\ref{formulation_inv_parfaite}) and relatively invisible obstacles. Again, we emphasize that for each problem one can imagine several formulations of the form (\ref{problem1}) with different functionals $F$, which are mathematically equivalent, but are not all well-suited for our method. 

\begin{remark}
	In practice, one can be interested in obstacles which act as modal converters: for a given incident field $w_m^+$, one wishes the energy to be completely transmitted on the mode $w_n^+$, $n\ne m$. In other words, we want to find $\rho$ such that $\G(\rho)=0$ with 
\[	
\G(\rho)=\left( ( \Re e\,R_{mp}^+,\Im m\,R_{mp}^+)_{0\le p\le N-1},( \Re e\,T_{mp}^+,\Im m\,T_{mp}^+)_{0\le n  \le N-1,\,p\neq n}\right).
\]
Observe that contrary to the previous examples, for this problem it is not simple to exhibit an initial $\rho_0$ such that $\G(\rho_0)=0$. Therefore our method cannot be used to construct modal converters.
\end{remark}

\subsection{The continuation method}\label{paragraphContinuation}

To solve (\ref{problem1}), we use a continuation method. We construct a sequence $(\rho_n)_{n\ge 1}$ such that for all integer $n$, we have $F(\rho_n)=F(\rho_0)$. Our objective is to obtain parameters $\rho_n$ which are quite different from $\rho_0$ with, for all $n\ge1$,  $F(\rho_n)=F(\rho_0)$. From a geometrical point of view, we move on the manifold $\{\rho\in\mL^{\infty}(\mathscr{O})\,|\,F(\rho)=F(\rho_0)\}$ starting from $\rho_0$. Note that in general this manifold is of infinite dimension.\\
\newline
Now, we explain how to construct from a solution $\rho_n$ of (\ref{problem1}) another $\rho_{n+1}$ such that $F(\rho_{n+1})=F(\rho_0)$. To set ideas, we focus our attention on functions $\rho_0$, $\rho_n$, $\rho_{n+1}$ which are supported in $\overline{\mathscr{O}}$ where $\mathscr{O}$ is a given non empty open subset of $\Om$. And we assume that $F$ is valued in $\R^d$, $d\ge1$. The idea consists in mimicking the proof of the implicit function theorem. We look for $\rho_{n+1}$ as a small perturbation of $\rho_n$. More precisely, we look for $\rho_{n+1}$ of the form $\rho_{n+1}=\rho_n+\eps\mu$, where $\eps>0$ is a small parameter and $\mu$ is a function of $\mL^\infty(\mathscr{O})$ to determine. Assuming that $F$ is continuously differentiable
, a Taylor expansion of $F$ at $\rho_n$ gives 
\begin{equation}\label{Taylor_def}
F(\rho_{n+1}) = F(\rho_n) + \eps \, dF(\rho_n)(\mu)  + \eps^2 \widetilde{F}^\eps(\mu),
\end{equation}
where $dF(\rho_n)(\mu)$ stands for the differential of $F$ at $\rho_n$ in the direction $\mu$ and where $\widetilde{F}^\eps(\mu)$ is an abstract remainder. Introduce the space 
\begin{equation}
\mathcal{N}(\rho_n) = \left\{ \mu  \in\mL^\infty(\mathscr{O}) \mid  dF(\rho_n)(\mu) = 0_{\R^d} \right\}=\ker\,dF(\rho_n).
\end{equation}
Since $dF(\rho_n):\mL^\infty(\mathscr{O})\to\R^d$ is a linear map and since $\mL^\infty(\mathscr{O})$ is of infinite dimension, $\mathcal{N}(\rho_n)$ is also of infinite dimension. In order to have $F(\rho_{n+1})$ close to $F(\rho_n)=F(\rho_0)$, a first idea is to take $\mu\in\mathcal{N}(\rho_n)$. In this case, there holds $|F(\rho_{n+1}) - F(\rho_0)|=O(\eps^2)$. However, in general we do not have $F(\rho_{n+1}) = F(\rho_0)$. In order to cancel the remainder, let us look for $\mu$ of the form $\mu = \mu_0 + \widetilde{\mu}$ with $\mu_0\in\mathcal{N}(\rho_n)\setminus\{0\}$ fixed and $\widetilde{\mu} \in\mL^\infty(\mathscr{O})$ to determine. Inserting this expression in (\ref{Taylor_def}), we get
\begin{equation}\label{presque_pt_fixe}
F(\rho_{n+1}) =F(\rho_n) = F(\rho_0)  \qquad\Leftrightarrow  \qquad dF(\rho_n)(\widetilde{\mu})  = - \eps \widetilde{F}^{\eps}(\mu_0+\widetilde{\mu}).
\end{equation}
Assume now, as for the implicit function theorem, that the differential  $dF(\rho_n):\mL^{\infty}(\mathscr{O})\to\R^d$ is onto. Then for $j=1,\dots,d$, we can find $\mu_j\in\mL^{\infty}(\mathscr{O})$ such that $dF(\rho_n)(\mu_j)=e_j$ where $(e_j)_{j=1}^d$ denotes the canonical basis of $\R^d$. Define the linear map
\[
\begin{array}{rccl}
K :&  \R^d & \to & \mathscr{\K}:=\mrm{span}(\mu_1,\dots,\mu_d)\subset \mL^\infty(\mathscr{O})\\[8pt]
 & \tau=(\tau_1,\dots,\tau_d) & \mapsto  & \dsp\K (\tau)=\sum_{j=1}^{d}\tau_j\mu_j.
\end{array}
\]
Note that $K$ is a right inverse for $dF(\rho_n)$, \textit{i.e.} we have $dF(\rho_n)\circ K = \mrm{Id}_{\R^d}$, and $ \K : \R^d \to \mathscr{\K}$ is a bijection. 
Now let us set 
\begin{equation}\label{chgt_espace} 
\widetilde{\mu} = \K (\tau), 
\end{equation}
where $\tau\in\R^d$ is to be determined. Inserting (\ref{chgt_espace}) in (\ref{presque_pt_fixe}), we get
\begin{equation}\label{pt_fixe} 
 \tau= - \eps \widetilde{F}^{\eps}\left(\mu_0 + \K(\tau) \right).
\end{equation}
This is a fixed point equation with respect to $\tau\in\R^d$. Now, when $F$ is continuously differentiable (that we will have to prove for our $F$), for any given $\radius>0$, for $\eps>0$ small enough, one can show that $\tau\mapsto - \eps \widetilde{F}^{\eps}\left(\mu_0 + \K (\tau) \right)$ is a contraction from $\overline{B(O,\radius)}$ to $\overline{B(O,\radius)}$ where $B(O,\radius)$ denotes the open ball of $\R^d$ centered at $O$ of radius $\radius$. The Banach fixed point theorem guarantees that (\ref{pt_fixe}) admits a unique solution $\tau^{\mrm{sol}}$ in $\overline{B(O,\radius)}$. Then for $\rho_{n+1}=\rho_n+\eps(\mu_0+\K(\tau^{\mrm{sol}}))$, we have $F(\rho_{n+1} ) = F(\rho_{n} ) = F(\rho_0)$. \newline In order to complete the description of the method, we have to check that $\rho_{n+1}\not\equiv\rho_n$. Assume by contradiction that $\rho_{n+1}\equiv\rho_n$. It means that $ \mu_0+\K(\tau^{\mrm{sol}})\equiv 0$. Applying $dF(\rho_n)$ to the latter equation, and using the fact that  $\mu_0\in\mathcal{N}(\rho_n)=\ker\,dF(\rho_n)$, we obtain $dF(\rho_n)\circ \K(\tau^{\mrm{sol}})=\tau^{\mrm{sol}}=0_{\R^d}$ and so $\mu_0\equiv0$, which is false, by hypothesis. \\
\newline
Summing up, we can state the following theorem.
\begin{theorem}	\label{the-ptfixe}
Assume that $F:\mL^{\infty}(\mathscr{O})\to \R^d$ is ${\mathscr C}^1$. Let $\rho\in \mL^{\infty}(\mathscr{O})$ be such that $dF(\rho)$ is onto. Let $\K$ be a right inverse of $dF(\rho)$ and $\mu_0$ be a non-trivial element of $\mathcal{N}(\rho)$. Then for all $\radius>0$, there is $\eps_0>0$ such that for all $\eps\in(0;\varepsilon_0]$
\[
\exists!\tau\in \overline{B(O,\radius)}\mbox{ such that }F(\rho+\varepsilon(\mu_0+\K(\tau)))=F(\rho).
\]
Moreover we have $\mu_0+\K(\tau)\not\equiv0$.
\end{theorem}
\begin{remark}
Observe that the obtained $\rho_{n+1}$ depends on $\eps$ and the $\mu_j$ and that the $\mu_j$ are not uniquely defined ($\mathcal{N}(\rho_n)$ is of infinite dimension). Note also that we could have replaced $\mL^\infty(\mathscr{O})$ by another subspace $\mathscr{E} \subset\mL^\infty(\mathscr{O})$. The only crucial point is that we need that the differential $ dF(\rho_n) :  \mathscr{E} \to \R^d$ to be onto. Choosing an appropriate subspace $\mathscr{E}$ of $\mL^\infty(\mathscr{O})$ will allow us to impose certain constraints to the obstacles we design (see Section \ref{SectionConstraints}).
\end{remark}
\begin{remark}
Let us clarify the connection with the implicit function theorem. Introduce the functional
\begin{equation}
\begin{array}{ccccc}
H  :  \mathcal{N}(\rho_n) \times  \mathscr{\K}  & \to & \R^d \\
\quad\qquad\ ( \mu , \eta ) & \mapsto & F( \rho_n + \mu + \eta) - F(\rho_0)
\end{array}
\end{equation}
which is of class $\mathscr{C}^1$. By definition of $\mathscr{\K} $, we know that $\partial_{\eta}H(0,0): \mathscr{\K}  \to \R^d$ is well-defined and bijective (this is a consequence of the identity $\partial_{\eta}H(0,0)\circ \K = \mrm{Id}_{\R^d}$). On the other hand, we remark that $H(0,0)=0$. The implicit function theorem applies: there are some neighbourhoods $\mathscr{V} \subset \mathcal{N}(\rho_n)$, $\mathscr{W} \subset  \mathscr{\K} $ of $0_{ \mathcal{N}(\rho_n) }$, $0_{ \mathscr{\K} }$ and a unique function $\varphi : \mathscr{V} \to \mathscr{W}$ of class $\mathscr{C}^1$ such that 
\[
[\,(\mu, \eta )\in \mathscr{V} \times \mathscr{W}\ \mbox{ and }\ H(\mu,\eta)=0\,] \qquad\Leftrightarrow\qquad \eta = \varphi (\mu).
\]
In particular, for all $\mu \in \mathcal{N}(\rho_n)$ close enough to $0_{ \mathcal{N}(\rho_n) }$, there is a unique $\eta = \varphi (\mu) \in \mathscr{\K}$ such that $F( \rho_n + \mu + \eta) = F(\rho_0)$.

\end{remark}

\section{Expression and differential of the scattering matrix}\label{SectionS}

The implementation of the continuation method presented in \S\ref{paragraphContinuation} depends on the properties of the scattering matrix and of its differential with respect to the material index. In this section, we compute these quantities. To proceed, we shall work with the symplectic (sesquilinear and anti-hermitian ($q(\varphi,\psi)=-\overline{q(\psi,\varphi)}$)) form $q(\cdot,\cdot)$ such that for all $\varphi,\psi\in\mH^1_{\loc}(\Om)$
\begin{equation}\label{DefFormq}
q(\varphi,\psi)=\dsp\int_{\Sigma_{-\ell}\cup\Sigma_{+\ell}}\cfrac{\partial \varphi}{\partial\nu}\,\overline{\psi}-\varphi\cfrac{\partial \overline{\psi}}{\partial\nu}\,d\sigma.
\end{equation}
Here we set $\Sigma_{\pm \ell}=\{\pm \ell\}\times(0;1)$ and $\partial_{\nu}=\pm\partial_x$ at $x=\pm \ell$. First, we obtain general formulas in the multimode regime. Then, in order to help the reader to get familiar with the different expressions, we write them explicitly in the simple situation where $N=1$ (monomode regime) in \S\ref{paragraphMono}.

\subsection{Expression of the scattering matrix}

\begin{proposition}\label{PropMatSca}
For $\rho\in\mL^{\infty}(\mathscr{O})$, the coefficients of the scattering matrix $\mathbb{S}(\rho)$ defined in (\ref{defScatteringMatrix}) are given by
\begin{equation}\label{DefSca}
R_{mn}^{\pm}=ik^2\dsp\int_{\Om}\rho u^{\pm}_{m}w^{\pm}_{n}\,dz,\qquad T_{mn}^{\pm}=\delta_{m,n}+ik^2\dsp\int_{\Om}\rho u^{\pm}_{m}w^{\mp}_{n}\,dz,\qquad 0 \le m,n\le N-1,
\end{equation}
where the $w^{\pm}_{n}$, $u^{\pm}_m$ are respectively defined in (\ref{defModes}), (\ref{scatteredFieldParticular}).
\end{proposition}
\begin{proof}
Start from the expansions (\ref{scatteredFieldParticular}) for $u_m^{\pm}$. Note in particular that decomposition in Fourier series guarantees that the evanescent parts $\tilde{u}_m^{\pm}$ expand only, in the $y$ direction, on the $\varphi_n$ (see (\ref{defCos})) such that $n\ge N$. Then using the particular normalisation of the modes $w_n^{\pm}$ in (\ref{defModes}) and the fact that the family $(\varphi_n)$ is orthonormal for the $\mL^2(0;1)$ inner product, we obtain for $0 \le m,n\le N-1$
\[
\begin{array}{ll}
\dsp\int_{\Sigma_{+\ell}}\cfrac{\partial w_m^{\pm}}{\partial\nu}\,\overline{w_n^{\pm}}-w_m^{\pm}\cfrac{\partial \overline{w_n^{\pm}}}{\partial\nu}\,d\sigma=\pm i\delta_{m,n},&\quad\dsp\int_{\Sigma_{-\ell}}\cfrac{\partial w_m^{\pm}}{\partial\nu}\,\overline{w_n^{\pm}}-w_m^{\pm}\cfrac{\partial \overline{w_n^{\pm}}}{\partial\nu}\,d\sigma=\mp i\delta_{m,n}\\[12pt]
\dsp\int_{\Sigma_{+\ell}}\cfrac{\partial w_m^{\pm}}{\partial\nu}\,\overline{w_n^{\mp}}-w_m^{\pm}\cfrac{\partial \overline{w_n^{\mp}}}{\partial\nu}\,d\sigma=0, &\quad\dsp\int_{\Sigma_{-\ell}}\cfrac{\partial w_m^{\pm}}{\partial\nu}\,\overline{w_n^{\mp}}-w_m^{\pm}\cfrac{\partial \overline{w_n^{\mp}}}{\partial\nu}\,d\sigma=0.
\end{array}
\]
Using these formulas, we get
\begin{equation}\label{FirstFourierCalculus}
\begin{array}{ll}
iR_{mn}^+=q(u_m^+,w_n^-),&\quad  i(T_{mn}^+-\delta_{m,n})=q(u_m^+,w_n^+),\\[5pt] 
  iR_{mn}^-=q(u_m^-,w_n^+),&\quad i(T_{mn}^--\delta_{m,n})=q(u_m^-,w_n^-).
\end{array}  
\end{equation}
Observing that $\overline{w_n^+}=w_n^-$ and integrating by parts in the above identities, we obtain 
\[
\begin{array}{ll}
iR_{mn}^+=\dsp\int_{\Om_{\ell}}\Delta u_m^+\,w_n^+- u_m^+\,\Delta w_n^+\,dz,&\quad  i(T_{mn}^+-\delta_{m,n})=\dsp\int_{\Om_{\ell}}\Delta u_m^+\,w_n^-- u_m^+\,\Delta w_n^-\,dz,\\[12pt] 
  iR_{mn}^-=\dsp\int_{\Om_{\ell}}\Delta u_m^-\,w_n^-- u_m^-\,\Delta w_n^-\,dz,&\quad i(T_{mn}^--\delta_{m,n})=\dsp\int_{\Om_{\ell}}\Delta u_m^-\,w_n^+- u_m^-\,\Delta w_n^+\,dz
\end{array}  
\]
with $\Om_{\ell}:=(-\ell;\ell)\times(0;1)$. This yields (\ref{DefSca}).
\end{proof}
\noindent Below we recall the proof of two classical features of the scattering matrix. 
\begin{proposition}\label{SymUnitary}
For $\rho\in\mL^{\infty}(\mathscr{O})$, the scattering matrix defined in (\ref{defScatteringMatrix}) is symmetric (\,$\mathbb{S}(\rho)^{\top}=\mathbb{S}(\rho)$\,) and unitary (\,$\mathbb{S}(\rho)\overline{\mathbb{S}(\rho)}^{\top}=\mrm{Id}^{2N\times 2N}$\,).
\end{proposition}
\begin{proof}
For $0\le m,n\le N-1$, consider the two functions $u_m^+$, $u_n^+$ defined in (\ref{scatteredFieldParticular}). Integrating by parts in the definition (\ref{DefFormq}) of $q(\cdot,\cdot)$ and using that $u_m^+$, $u_n^+$ solve the same problem (\ref{InitialPb}), one finds $q(u_m^+,\overline{u_n^+})=0$. On the other hand, a direct calculus (based on Fourier decomposition) similar to (\ref{FirstFourierCalculus}) gives $q(u_m^+,\overline{u_n^+})=i(R^{+}_{mn} -R^{+}_{nm})$. We deduce that $R^{+}_{mn}=R^{+}_{nm}$ and that $\mathcal{R}^+(\rho)$ is symmetric. Working in the same way on the quantities $q(u_m^-,\overline{u_n^-})$ and $q(u_m^{\pm},\overline{u_n^{\mp}})$, one can conclude that the whole matrix $\mathbb{S}(\rho)$ is symmetric.\\
\newline 
Now let us show that $\mathbb{S}(\rho)$ is unitary. Again, integrating by parts in the definition of $q(\cdot,\cdot)$, one gets $q(u_m^+,u_n^+)=0$. This times, a direct calculus based on Fourier decomposition gives 
\[
q(u_m^+,u_n^+)=i(-\delta_{m,n}+\sum_{j=0}^{N-1}R^{+}_{mj} \overline{R^{+}_{nj}}+T^{+}_{mj} \overline{T^{+}_{nj}}).
\]
Working similarly with $q(u_m^-,u_n^-)$, $q(u_m^{\pm},u_n^{\mp})$ and using that $\mathbb{S}(\rho)$ is symmetric, one deduces that $\mathbb{S}(\rho)\overline{\mathbb{S}(\rho)}^{\top}=\mrm{Id}^{2N\times 2N}$.
\end{proof}

\subsection{Differential of the scattering matrix}\label{ParagraphDiff}
\begin{proposition}\label{PropoDiffScaMat}
Assume that trapped modes do not exist for the problem (\ref{InitialPb}) at the considered $k$. Then the map $\mathbb{S}:\mL^{\infty}(\mathscr{O})\to \R^{2N\times 2N}$ is ${\mathscr C}^1$ in a neighbourhood of $\rho\in\mL^{\infty}(\mathscr{O})$ and the differential of $\mathbb{S}(\rho)$ in the direction $\mu\in\mL^{\infty}(\mathscr{O})$ is given by
\begin{equation}\label{ExpressionDifferentialsMulti}
 d\mathbb{S}(\rho)(\mu)=\left(\begin{array}{cc}
d\mathcal{R}^+(\rho)(\mu) & d\mathcal{T}^+(\rho)(\mu)\\[10pt]
d\mathcal{T}^-(\rho)(\mu) & d\mathcal{R}^-(\rho)(\mu)\\
\end{array}\right)\mbox{ with }\ \begin{array}{|l}
d\mathcal{R}^{\pm}(\rho)(\mu)=(ik^2\dsp\int_{\Om}\mu u^{\pm}_{m}u^{\pm}_{n}\,dz)_{0\le m,n\le N-1}\\[12pt]
d\mathcal{T}^{\pm}(\rho)(\mu)=(ik^2\dsp\int_{\Om}\mu u^{\pm}_{m}u^{\mp}_{n}\,dz)_{0\le m,n\le N-1},\end{array}\hspace{-0.2cm}
\end{equation}
where the $u^{\pm}_m$ are defined in (\ref{scatteredFieldParticular}). 
\end{proposition}
\begin{proof}
First we focus our attention on the computation of the coefficients of $d\mathcal{R}^+(\rho)(\mu)$ and $d\mathcal{T}^+(\rho)(\mu)$. To proceed, we consider the problem 
\begin{equation}\label{PerturbedEquation}
\begin{array}{|rlcl}
\multicolumn{4}{|l}{\mbox{Find }u^{+\eps}_m\mbox{ such that }u^{+\eps}_m-w^+_m\mbox{ is outgoing and}}\\[3pt]
\Delta u^{+\eps}_m+k^2(1+\rho+\eps\mu) u^{+\eps}_m&=&0&\mbox{ in }\Om\\[3pt]
\partial_y u^{+\eps}_m&=&0&\mbox{ on }\partial\Om.
\end{array}
\end{equation}
We make the ansatz $u^{+\eps}_m=u^0_m+\eps u^1_m+\dots$ where the dots stand for higher-order terms unimportant in our analysis. Plugging this expansion in (\ref{PerturbedEquation}) and identifying the powers in $\eps$ as $\eps\to0$, we find that $u^{0}_m$ and $u^{1}_m$ are respectively solutions to the problems 
\begin{equation}\label{LimitProblems}
\begin{array}{|rlcl}
\multicolumn{4}{|l}{\mbox{Find }u^{0}_m\mbox{ such that }u^{0}_m-w^+_m\mbox{ is outgoing and }}\\[3pt]
\Delta u^{0}_m+k^2(1+\rho) u^{0}_m&=&0&\mbox{ in }\Om\\[3pt]
\partial_y u^{0}_m&=&0&\mbox{ on }\partial\Om
\end{array}\,\begin{array}{|rlcl}
\multicolumn{4}{|l}{\mbox{Find }u^{1}_m\mbox{ such that }u^{1}_m\mbox{ is outgoing and }}\\[3pt]
\Delta u^{1}_m+k^2(1+\rho) u^{1}_m&=&-k^2\mu u^{0}_m&\mbox{ in }\Om\\[3pt]
\partial_y u^{1}_m&=&0&\mbox{ on }\partial\Om.
\end{array}\hspace{-0.4cm}
\end{equation}
From (\ref{LimitProblems}), first we deduce that we must set $u^0_m=u^+_m$ where $u^+_m$ is defined in (\ref{scatteredFieldParticular}). If we denote $R^{+\eps}_{mn}$, $T^{+\eps}_{mn}$ the scattering coefficients of $u^{+\eps}_m$, from the expansion $u^{+\eps}_m=u^0_m+\eps u^1_m+\dots$ together with the formula (\ref{FirstFourierCalculus}), we get $R^{+\eps}_{mn}=R^{+}_{mn}+\eps dR^+_{mn}(\rho)(\mu)+\dots$, $T^{+\eps}_{mn}=T^{+}_{mn}+\eps dT^+_{mn}(\rho)(\mu)+\dots$ with 
\begin{equation}\label{UseFormSympl}
i dR^+_{mn}(\rho)(\mu)=q(u^1_m,w^-_n),\qquad i dT^+_{mn}(\rho)(\mu)=q(u^1_m,w^+_n).
\end{equation}
On the other hand, a direct computation shows that 
\[
q(u^1_m,w^-_n)=q(u^1_m,\overline{u^+_n})\quad\mbox{ and }\quad q(u^1_m,w^+_n)=q(u^1_m,\overline{u^-_n}).
\]
This is due to the fact that $u^1_m$ is outgoing while $w^-_n$ and $\overline{u^+_n}$ (resp. $w^+_n$ and $\overline{u^-_n}$) have the same outgoing behaviour. Since $\Delta u^1_m +k^2(1+\rho)u^1_m=-k^2\mu u^+_m$ (see (\ref{LimitProblems})), integrating by parts in $q(u^1_m,\overline{u^\pm_n})$, we find
\begin{equation}\label{CalculdRdT}
i dR^+_{mn}(\rho)(\mu)=-k^2\dsp \int_{\Om}\mu u^+_mu^+_n\,dz\quad\mbox{ and }\quad i dT^+_{mn}(\rho)(\mu)=-k^2\dsp \int_{\Om}\mu u^+_m\,u^-_n\,dz.
\end{equation}
Besides, since $T^+_{mn}=T^-_{nm}$ (because $\mathbb{S}$ is symmetric), we have $dT^-_{mn}(\rho)(\mu)=dT^+_{nm}(\rho)(\mu)$. And working as for $dR^+_{mn}(\rho)(\mu)$, we can compute the expression of $dR^-_{mn}(\rho)(\mu)$. This gives (\ref{ExpressionDifferentialsMulti}).\\
\newline
Now we explain how to justify these formula. First, from (\ref{UseFormSympl}) we see that to prove that $\mathbb{S}:\mL^{\infty}(\mathscr{O})\to \R^{2N\times 2N}$ is differentiable in a neighbourhood of $\rho\in\mL^{\infty}(\mathscr{O})$, it is sufficient to establish the estimate
\begin{equation}\label{ErrorEstimate}
\|u^{+\eps}_m-(u^{+}_m+\eps u^1_m)\|_{\mH^1(\Om_{\ell})} \le C\,\eps^2
\end{equation}
with $C>0$ is independent of $\eps$. Define the classical Dirichlet-to-Neumann operators $\Lambda^{\pm}:\mH^{1/2}(\Sigma_{\pm\ell})\to\mH^{-1/2}(\Sigma_{\pm\ell})$ such that that for $\psi\in\mH^{1/2}(\Sigma_{\pm\ell})$, there  holds
\[
\Lambda^{\pm}(\psi)=\sum_{j=0}^{+\infty}\beta_j(\psi,\varphi_j)_{\mL^2(\Sigma_{\pm\ell})}\varphi_j.
\]
Then with the Riesz representation theorem, introduce the operator $A^{\eps}(\rho):\mH^1(\Om_{\ell})\to\mH^1(\Om_{\ell})$ such that for all $\psi$, $\psi'\in\mH^1(\Om_{\ell})$, 
\[
(A^{\eps}(\rho)\psi,\psi')_{\mH^1(\Om_{\ell})}=\int_{\Om_{\ell}}\nabla\psi\cdot\overline{\nabla\psi'}-k^2(1+\rho+\eps\mu)\psi\overline{\psi'}\,dz-\langle \Lambda^+(\psi),\overline{\psi'} \rangle_{\Sigma_{+\ell}}-\langle \Lambda^-(\psi),\overline{\psi'} \rangle_{\Sigma_{-\ell}}.
\]
Here $(\cdot,\cdot)_{\mH^1(\Om_{\ell})}$ stands for the inner product of $\mH^1(\Om_{\ell})$ and $\langle \cdot,\cdot\rangle_{\Sigma_{\pm\ell}}$ denotes the (linear) duality product $\mH^{-1/2}(\Sigma_{\pm\ell})\times \mH^{1/2}(\Sigma_{\pm\ell})$. Using the assumption that trapped modes do not exist for the problem (\ref{InitialPb}), we infer that $A^{0}(\rho):\mH^1(\Om_{\ell})\to\mH^1(\Om_{\ell})$ is an isomorphism. Observing that $A^{\eps}(\rho)-A^{0}(\rho)$ is small in operator norm for $\eps$ sufficiently small, we deduce that $A^{\eps}(\rho)$ is invertible for $\eps$ sufficiently small. And there exists $\eps_0>0$ such that for all $\eps\in(0;\eps_0]$, we have the stability estimate
\begin{equation}\label{UniformBound}
\|(A^{\eps}(\rho))^{-1}\|\le C,
\end{equation}
where $C$ is independent of $\eps$. Using (\ref{UniformBound}) and observing that $e^{\eps}:=u^{+\eps}_m-(u^{+}_m+\eps u^1_m)$ solves the problem
\[
\begin{array}{|rlcl}
\multicolumn{4}{|l}{\mbox{Find }e^{\eps}\mbox{ such that }e^{\eps}\mbox{ is outgoing and}}\\[3pt]
\Delta e^{\eps}+k^2(1+\rho+\eps\mu) e^{\eps}&=&-\eps^2k^2\mu u^1_m&\mbox{ in }\Om\\[3pt]
\partial_y e^{\eps}&=&0&\mbox{ on }\partial\Om,
\end{array}
\]
we obtain the error estimate (\ref{ErrorEstimate}). Denote 
$u^{0}_m(\tilde{\rho})$, $u^{1}_m(\tilde{\rho})$ the solutions of (\ref{LimitProblems}) with $\rho$ replaced by $\tilde{\rho}$ close to $\rho$ for the norm of $\mL^{\infty}(\mathscr{O})$. Using again results of perturbations of operators, first we establish that for $\tilde{\rho}$ close enough to $\rho$, we have $\|u^{0}_m(\rho)-u^{0}_m(\tilde{\rho})\|_{\mH^1(\Om_{\ell})}\le C\,\|\rho-\tilde{\rho}\|_{\mL^{\infty}(\mathscr{O})}$. We deduce that 
$\|u^{1}_m(\rho)-u^{1}_m(\tilde{\rho})\|_{\mH^1(\Om_{\ell})} \le C\,\|\rho-\tilde{\rho}\|_{\mL^{\infty}(\mathscr{O})}\|\mu\|_{\mL^{\infty}(\mathscr{O})}$. This is enough to conclude that $\rho\mapsto dS(\rho)$ is continuous from $\mL^{\infty}(\mathscr{O})$ to $\mathcal{L}(\mL^{\infty}(\mathscr{O}),\R^{2N\times 2N})$. This ends to show that $\mathbb{S}:\mL^{\infty}(\mathscr{O})\to \R^{2N\times 2N}$ is ${\mathscr C}^1$ in a neighbourhood of $\rho\in\mL^{\infty}(\mathscr{O})$.
\end{proof}
\noindent In the next proposition, we prove that the structure of the scattering matrix translates into a structure for the scattering solutions. 
\begin{proposition}
Set $U:=(u_0^+,\dots,u_{N-1}^+,u_0^-,\dots,u_{N-1}^-)^{\top}$ where the $u_{m}^{\pm}$ are the scattering solutions introduced in (\ref{scatteredFieldParticular}). We have the identity
\begin{equation}\label{ImportantRelation}
\overline{\mathbb{S}(\rho)}U=\overline{U}. 
\end{equation}
\end{proposition}
\begin{proof}
Looking at the behaviour for $|x|\ge \ell$ and using the fact that $\mathbb{S}(\rho)$ is unitary, one finds that $\overline{\mathbb{S}(\rho)}U-\overline{U}$ is a vector of functions which solve the homogeneous problem (\ref{InitialPb}) and which are exponentially decaying at infinity. In other words, $\overline{\mathbb{S}(\rho)}U-\overline{U}$ is a vector of trapped modes. But since by definition the $u^{\pm}_m$ are orthogonal to trapped modes for the $\mL^2(\Om)$ inner product, we deduce (\ref{ImportantRelation}).
\end{proof}
\noindent Using relation (\ref{ImportantRelation}) in (\ref{ExpressionDifferentialsMulti}), we get the following statement which will be useful to address the problem of relative invisibility (see \S\ref{paragraphRelativeInv}).
\begin{proposition}
Assume that trapped modes do not exist for the problem (\ref{InitialPb}) at the considered $k$. Then for $\rho,\,\mu\in\mL^{\infty}(\mathscr{O})$, we have the identity
\[
\overline{\mathbb{S}(\rho)}\,d\mathbb{S}(\rho)(\mu)=ik^2\left(\begin{array}{cc}
\Big(\dsp\int_{\Om}\mu \overline{u_m^+}\,u_n^+ \,dz\Big)_{0\le m,n\le N-1 } & \Big(\dsp\int_{\Om}\mu \overline{u_m^+}\,u_n^- \,dz\Big)_{0\le m,n\le N-1 }\\[12pt]
\Big(\dsp\int_{\Om}\mu \overline{u_m^-}\,u_n^+ \,dz)_{0\le m,n\le N-1 } & \Big(\dsp\int_{\Om}\mu \overline{u_m^-}\,u_n^- \,dz\Big)_{0\le m,n\le N-1 }
\end{array}\right).
\]
\end{proposition}

\subsection{Monomode regime}\label{paragraphMono}
When the wavenumber $k$ is such that $0<k<\pi$, only the 
mode $w_0^{\pm}$ can propagate in the waveguide $\Om$ ($N=1$). This monomode regime will play an important role in our analysis later. In particular, we will be able to prove stronger results than in the case $N\ge2$. Moreover, it will allow us to understand more easily why certain choices of functional $F$ in (\ref{problem1}) are not adapted. To simplify, when $N=1$, we shall denote $w^{\pm}$, $u^{\pm}$, $R^{\pm}$, $T^{\pm}$ instead of $w^{\pm}_0$, $u^{\pm}_0$, $R^{\pm}_{00}$, $T^{\pm}_{00}$ respectively. Since the scattering matrix is symmetric, we shall set $T=T^+=T^-$ so that, with the help of Proposition \ref{PropMatSca}, we can write
\[
\mathbb{S}(\rho)=\left(\begin{array}{cc}
R^+ & T\\
T & R^-\\
\end{array}\right)=\left(\begin{array}{cc}
ik^2\dsp\int_{\Om}\rho u^+w^+\,dz & 1+ik^2\dsp \int_{\Om}\rho u^+w^-\,dz\\[10pt]
1+ik^2\dsp\int_{\Om}\rho u^-w^+\,dz & ik^2\dsp \int_{\Om}\rho u^-w^-\,dz\\
\end{array}\right).
\]
The unitarity of $\mathbb{S}(\rho)$ is equivalent to the following three identities
\begin{equation}\label{Unitary1D}
|R^+|^2+|T|^2=1;\qquad |R^-|^2+|T|^2=1;\qquad \overline{R^+}T+\overline{T}R^-=0.
\end{equation}
On the other hand, for the differential of the scattering matrix, we have the formulas
\begin{equation}\label{DiffSca1D}
d\mathbb{S}(\rho)(\mu)=ik^2\hspace{-0.1cm}\left(\begin{array}{cc}
\hspace{-0.15cm}\dsp\int_{\Om}\hspace{-0.15cm}\mu (u^+)^2dz & \hspace{-0.2cm}\dsp\int_{\Om}\hspace{-0.15cm}\mu u^+u^-dz\\[10pt]
\hspace{-0.15cm}\dsp\int_{\Om}\hspace{-0.15cm}\mu u^-u^+dz & \hspace{-0.2cm}\dsp\int_{\Om}\hspace{-0.15cm}\mu (u^-)^2dz\\
\end{array}\hspace{-0.15cm}\right);\quad\overline{\mathbb{S}(\rho)}\,d\mathbb{S}(\rho)(\mu)=ik^2\hspace{-0.1cm}\left(\begin{array}{cc}
\hspace{-0.15cm}\dsp\int_{\Om}\hspace{-0.15cm}\mu |u^+|^2dz & \hspace{-0.2cm}\dsp\int_{\Om}\hspace{-0.15cm}\mu \overline{u^+}u^-dz\\[10pt]
\hspace{-0.15cm}\dsp\int_{\Om}\hspace{-0.15cm}\mu\overline{u^-}u^+dz & \hspace{-0.2cm}\dsp\int_{\Om}\hspace{-0.15cm}\mu |u^-|^2dz\\
\end{array}\hspace{-0.15cm}\right)\hspace{-0.1cm}.\hspace{-0.1cm}
\end{equation}
\begin{remark}\label{RemarkdNRJ}
Note that the second identities of (\ref{DiffSca1D}) imply in particular $\Re e\,(\overline{R^{\pm}(\rho)}dR^{\pm}(\rho)(\mu)+\overline{T(\rho)}dT(\rho)(\mu))=0$. These results can be obtained directly by differentiating the equations of (\ref{Unitary1D}). 
\end{remark}
\noindent Relation (\ref{ImportantRelation}) becomes
\begin{equation}\label{ImportantRelation1D}
\overline{\mathbb{S}(\rho)}U=\overline{U}\quad\Leftrightarrow\quad \left(\begin{array}{cc}
\overline{R^+} & \overline{T}\\
\overline{T} & \overline{R^-}
\end{array}\right)\left(\begin{array}{cc}
u^+\\
u^-
\end{array}\right)=\left(\begin{array}{cc}
\overline{u^+}\\
\overline{u^-}
\end{array}\right)\quad\Leftrightarrow\quad \begin{array}{|l}
\overline{R^+}u^++\overline{T}u^-=\overline{u^+} \\
\overline{T}u^++\overline{R^-}u^-=\overline{u^-}.
\end{array} 
\end{equation}
In the next proposition, we state a result which will be useful to study the ontoness of some functionals $F$ below. 

\begin{proposition}\label{propositionImagMono}
Assume that $\rho$ is such that $T=0$. Then the functions $\Re e\,u^+$ and $\Im m\,u^+$ (respectively $\Re e\,u^-$ and $\Im m\,u^-$) are linearly dependent in $\Om$.
\end{proposition}
\begin{proof}
Assume that $T=T(\rho)=0$. Then by conservation of energy (\ref{Unitary1D}), we have $|R^+|=1$ and there is $\theta\in[0;2\pi)$ such that $R^+=e^{i\theta}$. From (\ref{ImportantRelation1D}), we deduce that $e^{-i\theta/2}u^+=\overline{e^{-i\theta/2}u^+}$ and so $\Im m\,(e^{-i\theta/2}u^+)\equiv0$. This shows that the functions $\Re e\,u^+$ and $\Im m\,u^+$ are linearly dependent. The proof is similar for $\Re e\,u^-$ and $\Im m\,u^-$.
\end{proof}

\begin{remark}
Note that if $\rho$ is such that $T(\rho)=0$, then the fields $(t,x,y)\mapsto\Re e\,(u^{\pm}(x,y)e^{-i\om t})$ are stationary vibration modes in $\Om$.
\end{remark}

\section{Justification of the continuation method in monomode regime}\label{SectionMono}
 
We come back to the three problems introduced in \S\ref{paragraphProblems}, namely invisibility in reflection, invisibility in reflection and transmission, and relative invisibility. For each case, for the functional $F$ introduced in \S\ref{paragraphProblems}, we shall study the ontoness of the differential of $F$. We remind the reader that this property, as explained in \S\ref{paragraphContinuation} (see the discussion after (\ref{presque_pt_fixe})), is the corner stone of the continuation method. 
We will work exclusively in monomode regime ($N=1$) when $0<k<\pi$. When $\rho\equiv0$, we simply have $u^{\pm}=w^{\pm}$ in (\ref{DiffSca1D}) so that we get explicit formula for $dF(0)$ (see Remark \ref{RmkExpliDiff} below). Then we can show directly that $dF(0)$ is onto. But when $\rho\not\equiv0$, it is necessary to develop a more abstract analysis. In this work, we give complete proofs in monomode regime. At higher wavenumber when $N\ge2$, results of ontoness of the functionals seem harder to establish and their derivation is still an open problem.

\subsection{A preliminary lemma}
The different results of ontoness of the differential that will be proved below will make use of the following lemma. We remind the reader that $\overline{\mathscr{O}}$ corresponds to the support of the obstacle.
\begin{lemma}	\label{lem-Q(X,Y)}
Let $Q(X,Y)=\alpha X^2+2\beta XY+\gamma Y^2$ be a quadratic form on $\R^2$  
with $\alpha,\beta,\gamma\in\R$ such that $(\alpha,\beta,\gamma)\neq (0,0,0)$.  
Assume that $u$ and $v$ are two real valued solutions of (\ref{InitialPb}). If \begin{equation}\label{QO}
Q(u(x,y),v(x,y))=0\qquad\forall (x,y)\in \mathscr{O},
\end{equation} then the same identity holds for $(x,y)\in \Omega$ and there exists a non trivial linear combination of $u$ and $v$  which vanishes in $\Omega$.
\end{lemma}
\begin{proof}
Suppose first that $\alpha=\gamma=0$. Then, necessarily $\beta\neq 0$ and  $u(x,y)v(x,y)=0$ for   all $(x,y)\in\mathscr{O}$. Since $u$ and $v$ are continuous functions, at least one of them vanishes on an open subset of $\mathscr{O}$, and then everywhere in $\Omega$ by the unique continuation principle (see \S8.3 of \cite{CoKr13} and the references therein).
\newline
Now we suppose without loss of generality that $\alpha\neq 0$. In that case, we can write
\[
Q(X,Y)=\alpha^{-1}\left[\left(\alpha X+\beta Y\right)^2+(\alpha\gamma-\beta^2)Y^2\right].
\]
Then there are two possibilities.\\
\newline
$\bullet$ If $\alpha\gamma-\beta^2\geq 0$, $Q(X,Y)=0$ implies $\alpha X+\beta Y=0$ and $\dsp (\alpha\gamma-\beta^2)Y=0$. From (\ref{QO}), we deduce that $\alpha u+\beta v$ and $\dsp (\alpha\gamma-\beta^2)v=0$ vanish  identically in $\mathscr{O}$, and then everywhere in $\Omega$ by the unique continuation principle. The result follows.\\
\newline
$\bullet$ If $\alpha\gamma-\beta^2< 0$,  
\[
Q(X,Y)=\alpha^{-1}\Big(\alpha X+(\beta+\sqrt{\beta^2-\alpha\gamma}\,)Y\Big)\Big(\alpha X+(\beta-\sqrt{\beta^2-\alpha\gamma}\,)Y\Big).
\]
Then condition (\ref{QO}) implies that $(\alpha u+(\beta+\sqrt{\beta^2-\alpha\gamma}\,)v)(\alpha u+(\beta-\sqrt{\beta^2-\alpha\gamma}\,)v)$ vanishes in $\mathscr{O}$. It is again the product of two functions satisfying (\ref{InitialPb}). As a consequence, they are continuous and at least one of them vanishes on an open subset of $\mathscr{O}$. This implies that it vanishes everywhere in $\Omega$ by the unique continuation principle. And the lemma follows. 
\end{proof}
\subsection{Invisibility in reflection}\label{ParagRNullMono}
In monomode regime, the reflection matrix $\mathcal{R}^+(\rho)$ is nothing but the complex number $R^+(\rho)$. In this paragraph, we wish to find functions $\rho\in\mL^{\infty}(\mathscr{O})$ such that $R^+(\rho)=0$. To proceed, we said in (\ref{NoReflection}) that we can work with $\rho_0\equiv0$ and 
\begin{equation}\label{DefFTNull1D}
F(\rho)=( \Re e\,R^+(\rho), \Im m\,R^+(\rho))\in\R^2.
\end{equation}
\begin{proposition}\label{DifferentialR}
Set $0<k<\pi$ ($N=1$). The map $dF(\rho):\mrm{L}^{\infty}(\mathscr{O})\to\R^2$ with $F$ defined in (\ref{DefFTNull1D}) is onto if and only if $\rho\in\mrm{L}^{\infty}(\mathscr{O})$ is such that $T(\rho)\ne0$.
\end{proposition}
\begin{remark}\label{RmkExpliDiff}
Note that for $\rho\equiv0$, we have $u^{\pm}=w^{\pm}$ in (\ref{DiffSca1D}) and so 
\[
dF(0)(\mu)=k\,(-\int_{\Om}\mu\sin(2kx)\,dz,\int_{\Om}\mu\cos(2kx)\,dz)/2.
\]
Then it is clear that $dF(0):\mrm{L}^{\infty}(\mathscr{O})\to\R^2$ is onto.
\end{remark}
\begin{proof}
Formula (\ref{DiffSca1D}) guarantees that for $\mu\in\mrm{L}^{\infty}(\mathscr{O})$, we have
\[
dR^+(\rho)(\mu)=ik^2\dsp\int_{\Om}\mu (u^+)^2\,dz .
\]
From Lemma \ref{lemmaGram} below, we infer that $dF(\rho):\mrm{L}^{\infty}(\mathscr{O})\to\R^2$ is onto if and only if $\{\Re e\,((u^+)^2),\Im m\,((u^+)^2)\}$ is a family of linearly independent functions. Assume that we have 
\begin{equation}\label{Freedom}
\alpha\,\Re e\,((u^+)^2)+\beta\,\Im m\,((u^+)^2)=0\quad \mbox{ in }\mathscr{O}
\end{equation}
for some constants $\alpha$, $\beta\in\R$.\\ 
\newline
$\star$ First, we study the case $T(\rho)\ne0$. Set $a=\Re e\,u^+$ and $b=\Im m\,u^+$. Since $\Re e\,((u^+)^2)=a^2-b^2$ and $\Im m\,((u^+)^2)=2a b$, relation (\ref{Freedom}) implies $\alpha\,(a^{2}-b^2)+\beta\,2ab=0$ in $\mathscr{O}$. From Lemma \ref{lem-Q(X,Y)}, we deduce that if $(\alpha,\beta)\ne(0,0)$, then there are some constants $A$, $B$ with $(A,B)\ne(0,0)$ such that
\[
A\,a+B\,b=0\mbox{ in }\Om.
\]
Then there would exist $\theta\in[0;2\pi)$ such that $e^{i\theta}u^+$ is purely real in $\Om$. Since $u^+$ admits the expansion 
\begin{equation}\label{ExpansionInf}
u^+=T(\rho)\,e^{ikx}+\tilde{u}^+\qquad\mbox{ for }x\ge \ell,
\end{equation}
this is impossible when $T(\rho)\ne0$ (we remind the reader that $\tilde{u}^+$ is exponentially decaying). Thus if $T(\rho)\ne0$, then we must have $\alpha=\beta=0$ in (\ref{Freedom}) which guarantees that $dF(\rho)$ is onto.\\
\newline
$\star$ Now we prove that $dF(\rho):\mL^{\infty}(\mathscr{O})\to\R^2$ is not onto when $T(\rho)=0$. When $T(\rho)=0$, Proposition \ref{propositionImagMono} ensures that the functions $\Re e\,u^+$ and $\Im m\,u^+$ are linearly dependent. As a consequence, there are some constants $\eta$, $\gamma\in\R$ with $(\eta,\gamma)\ne(0,0)$ such that $\eta\,a+\gamma\,b=0$ in $\Om$. This allows one to show that there is a pair $(\alpha,\beta)\ne(0,0)$ such that $\alpha\,(a^{2}-b^2)+\beta\,2ab=0$.
\end{proof}
\begin{remark}
Another way to prove the second item of the proof above is to observe from Remark \ref{RemarkdNRJ} that we have $\Re e\,(\overline{R^{+}(\rho)}dR^{+}(\rho)(\mu))=0$ for all $\mu\in\mrm{L}^{\infty}(\mathscr{O})$ when $T(\rho)=0$. Since $|R^{+}(\rho)|=1$ when $T(\rho)=0$, we deduce that  $dF(\rho):\mrm{L}^{\infty}(\mathscr{O})\to\R^2$ is not onto in this case. This is similar to what is represented on Figures \ref{FigureConservationNRJ}, \ref{FigureConservationNRJUnivseral} if one inverts the roles of $R^+$ and $T$.
\end{remark}
\noindent Before proceeding, we show a technical result needed in the above proof.
\begin{lemma}\label{lemmaGram}
Let $f_1,\dots,f_d$, $d\ge1$, be real valued functions of $\mL^{\infty}(\mathscr{O})$. Then the map 
\begin{equation}\label{mapOnto}
\mu\mapsto (\int_{\mathscr{O}}\mu f_1\,dz,\dots,\int_{\mathscr{O}}\mu f_d\,dz)
\end{equation}
from $\mL^{\infty}(\mathscr{O})$ to $\R^d$ is onto if and only if $\{f_1,\dots,f_d\}$ is a family of linearly independent functions.
\end{lemma}
\begin{proof}
Since $\mL^{\infty}(\mathscr{O})\subset\mL^2(\mathscr{O})$, every function of $\mL^{\infty}(\mathscr{O})$ decomposes as an element of $\mrm{span}(f_1,\dots,f_d)$ plus another element in the kernel of the map (\ref{mapOnto}). Therefore, the map (\ref{mapOnto}) is onto from $\mL^{\infty}(\mathscr{O})$ to $\R^d$ if and only if it is onto from $\mrm{span}(f_1,\dots,f_d)$ to $\R^d$. This is true if and only if it is injective in $\mrm{span}(f_1,\dots,f_d)$, which is equivalent to the fact that the matrix $(\int_{\mathscr{O}}f_if_j\,dz)_{1\le i,j\le d}$ is invertible. And clearly $(\int_{\mathscr{O}}f_if_j\,dz)_{1\le i,j\le d}$ is invertible if and only if $\{f_1,\dots,f_d\}$ is a family of linearly independent functions.
\end{proof}

\noindent From Theorem \ref{the-ptfixe} as well as Propositions  
\ref{PropoDiffScaMat} and \ref{DifferentialR}, we deduce the following statement. Here $\mathscr{\K}=\mrm{span}(\mu_1,\mu_2)$ is a subspace of $\mrm{L}^{\infty}(\mathscr{O})$ of dimension $2$ such that $dF(\rho):\mathscr{\K}\to\R^2$ is a bijection. 
\begin{theorem}\label{NonRef1D}
Set $0<k<\pi$ ($N=1$). Assume that $\rho\in\mrm{L}^{\infty}(\mathscr{O})$ is such that $R^+(\rho)=0$ and that trapped modes do not exist for the  problem (\ref{InitialPb}). Let $\mu_0$ be a non-trivial element of $dF(\rho)$. Then for all $\radius>0$, there is $\eps_0>0$ such that for all $\eps\in(0;\varepsilon_0]$
\[
\exists!\tau=(\tau_1,\tau_2)\in \overline{B(O,\radius)}\subset\R^2\mbox{ such that }R^+(\rho+\varepsilon(\mu_0+\tau_1\mu_1+\tau_2\mu_2))=0.
\]
\end{theorem}
\begin{remark}
Once a $\rho^{n+1}$ such that $R^+(\rho^{n+1})=0$ has been constructed from a $\rho^{n}$ such that $R^+(\rho^{n})=0$, we can iterate the process thanks to the previous theorem. This continuation method allows us to get non reflecting obstacles with large amplitudes as we will see in the numerics of Section \ref{SectionNumerics}.
\end{remark}
\begin{remark}
Note that since $\mathbb{S}(\rho)$ is unitary, we have $R^+(\rho)=0\Leftrightarrow R^-(\rho)=0$.
\end{remark}

\subsection{Invisibility in reflection and transmission}\label{paragraphInvRefTrans}

In this paragraph, we wish to find functions $\rho\in\mL^{\infty}(\mathscr{O})$ such that $\mathbb{S}(\rho)=\mrm{Id}^{2\times2}$. To proceed, we said in (\ref{formulation_inv_parfaite}) that we can take $\rho_0\equiv0$ and 
\begin{equation}\label{DefFTOne1D}
F(\rho)=( \Re e\,R^+(\rho), \Im m\,R^+(\rho),\Im m\,T(\rho))\in\R^3.
\end{equation}

\begin{proposition}\label{DifferentialT}
Set $0<k<\pi$ ($N=1$). The map $dF(\rho):\mrm{L}^{\infty}(\mathscr{O})\to\R^3$ with $F$ defined in (\ref{DefFTOne1D}) is onto if and only if $\rho\in\mrm{L}^{\infty}(\mathscr{O})$ is such that $\Re e\,T(\rho)\ne0$.
\end{proposition}
\begin{proof}
Formula (\ref{DiffSca1D}) guarantees that for $\mu\in\mrm{L}^{\infty}(\mathscr{O})$, we have
\[
dR^+(\rho)(\mu)=ik^2\dsp\int_{\Om}\mu (u^+)^2\,dz\qquad\mbox{ and }\qquad dT(\rho)(\mu)=ik^2\dsp\int_{\Om}\mu u^+ u^-\,dz.
\]
Using Lemma \ref{lemmaGram}, we deduce that $dF(\rho):\mrm{L}^{\infty}(\mathscr{O})\to\R^3$ is onto if and only if the family of functions $\{\Re e\,((u^+)^2),\Im m\,((u^+)^2),\Re e\,(u^+u^-)\}$ is linearly independent.\\
$\star$ If $T(\rho)=0$, then Proposition \ref{DifferentialR} ensures that $dR^+(\rho):\mrm{L}^{\infty}(\mathscr{O})\to\Cplx$ is not onto. In this case, $dF(\rho):\mrm{L}^{\infty}(\mathscr{O})\to\R^3$ cannot be onto.\\
$\star$ Assume now that $\Re e\,T(\rho)=0$ with $T(\rho)\ne0$. Formula (\ref{ImportantRelation1D}) implies the identity $u^-=(\overline{u^+}-\overline{R^+}u^+)/\overline{T}$. Therefore, we have 
\begin{equation}\label{eqnOnto1}
\Re e\,(u^+u^-) = \Re e\,( |u^+|^2/\overline{T}-\overline{R^+}(u^+)^2/\overline{T}).
\end{equation}
As a consequence, we see that when $T\in\R i\setminus\{0\}$, we have $\Re e\,(u^+u^-)\in\mrm{span}(\Re e\,((u^+)^2),\Im m\,((u^+)^2))$ and $\{\Re e\,((u^+)^2),\Im m\,((u^+)^2),\Re e\,(u^+u^-)\}$ is a family of linearly dependent functions.\\
$\star$ Finally we consider the case $\Re e\,T\ne0$. Assume that there are some real constants $\alpha$, $\beta$, $\gamma$ such that 
\begin{equation}\label{eqnOnto2}
\alpha\,\Re e\,((u^+)^2)+\beta\,\Im m\,((u^+)^2)+\gamma\,\Re e\,(u^+u^-)=0\quad \mbox{ in }\mathscr{O}.
\end{equation}
Next, we remove the dependence with respect to $u^-$ using the formulas of (\ref{ImportantRelation1D}) in order to apply Lemma \ref{lem-Q(X,Y)}. More precisely, inserting (\ref{eqnOnto1}) in (\ref{eqnOnto2}) and setting again $a=\Re e\,u^+$, $b=\Im m\,u^+$, we find that there are some real constants $A\ne0$, $B$, $C$ such that
\[
\alpha\,(a^{2}-b^2)+\beta\,2ab+\gamma\,( A\,(a^2+b^2)+B\,(a^{2}-b^2)+C\,2ab)=0\quad \mbox{ in }\mathscr{O}.
\]
This implies $(\alpha+\gamma\,(A+B))\,a^{2}+(\beta+\gamma\,C)\,2ab+(-\alpha+\gamma\,(A-B))\,b^{2}=0$. Working as in the proof of Proposition \ref{DifferentialR}, we obtain $\alpha+\gamma\,(A+B)=0$, $\beta+\gamma\,C=0$ and $-\alpha+\gamma\,(A-B)=0$. Since $A\ne0$, we deduce that $\gamma=0$ and $\alpha=\beta=0$. Thus $\{\Re e\,((u^+)^2),\Im m\,((u^+)^2),\Re e\,(u^+u^-)\}$ is a family of linearly independent functions.
\end{proof}
\begin{remark}\label{RmkDiffNRJ}
Another way to establish the second item of the proof of Proposition \ref{DifferentialT} is to use again the identity $\Re e\,(\overline{R^+(\rho)}dR^+(\rho)(\mu))+\Re e\,(\overline{T(\rho)}dT(\rho)(\mu))=0$ (see Remark \ref{RemarkdNRJ}). As a consequence, when $\Re e\,T(\rho)=0$ with $T(\rho)\ne0$, we deduce that there are some real $A$, $B$ independent of $\mu\in\mL^{\infty}(\mathscr{O})$ such that 
\[
\Im m\,(dT(\rho)(\mu))=A\,\Re e\,(dR^+(\rho)(\mu))+B\,\Im m\,(dR^+(\rho)(\mu)).
\]
This allows us to conclude that the  map $\mu\mapsto (\Re e\,(dR^+(\rho)(\mu)),\Im m\,(dR^+(\rho)(\mu)),\Im m\,(dT(\rho)(\mu)))$ is not onto in $\R^3$. 
\end{remark}

\noindent From Theorem \ref{the-ptfixe} as well as Propositions  
\ref{PropoDiffScaMat} and \ref{DifferentialT}, we deduce the following statement. Here $\mathscr{\K}=\mrm{span}(\mu_1,\mu_2,\mu_3)$ is a subspace of $\mrm{L}^{\infty}(\mathscr{O})$ of dimension $3$ such that $dF(\rho):\mathscr{\K}\to\R^3$ is a bijection. 
\begin{theorem}\label{PerfectTrans1D}
Set $0<k<\pi$ ($N=1$). Assume that $\rho\in\mrm{L}^{\infty}(\mathscr{O})$ is such that $\mathbb{S}(\rho)=\mrm{Id}^{2\times2}$ and that trapped modes do not exist for the problem (\ref{InitialPb}). Let $\mu_0$ be a non-trivial element of $dF(\rho)$. Then for all $\radius>0$, there is $\eps_0>0$ such that for all $\eps\in(0;\varepsilon_0]$
\[
\exists!\tau=(\tau_1,\tau_2,\tau_3)\in \overline{B(O,\radius)}\subset\R^3\mbox{ such that }\mathbb{S}(\rho+\varepsilon(\mu_0+\sum_{j=1}^3\tau_j\mu_j))=\mrm{Id}^{2\times2}.
\]
\end{theorem}

\begin{proof}
From the analysis of \S\ref{paragraphContinuation} and the result of Proposition \ref{DifferentialT}, we know that for $\eps$ small enough, there is a unique $\tau=(\tau_1,\tau_2,\tau_3)\in \overline{B(O,\radius)}\subset\R^3$ such that
\[
F(\rho+\varepsilon(\mu_0+\sum_{j=1}^3\tau_j\mu_j))=0.
\]
Set $\eta:=\varepsilon(\mu_0+\sum_{j=1}^3\tau_j\mu_j)$. By definition (\ref{DefFTOne1D}) of $F$, then we have $R^+(\rho+\eta)=0$ and $\Im m\,T(\rho+\eta)=0$. 
By conservation of energy (\ref{Unitary1D}), we must have $|R^+|^2+|T|^2=1$. We deduce that either $T(\rho+\eta)=1$ or $T(\rho+\eta)=-1$. However, since $\eta$ is small, $T(\rho+\eta)$ is close to $T(\rho)=1$. We infer that $T(\rho+\eta)=1$. 
\end{proof}
\noindent In the rest of this paragraph, we explain why the choice of the functional $F$ defined in (\ref{DefFTOne1D}) is the most relevant one to impose $T=1$. As already mentioned in \S\ref{paragraphProblems}, a seemingly more economic idea would have been to set $F(\rho)=\Re e\,T$. Indeed if $\rho_0$ is such that $T(\rho_0)=1$ and if $\rho$ is close to $\rho_0$ with $F(\rho)=F(\rho_0)$, then we also have $T(\rho)=1$. The problem with this approach is that $dF(\rho_0):\mL^{\infty}(\mathscr{O})\to\R$ is not onto. And more precisely, we have $dF(\rho_0)(\mu)=0$ for all $\mu\in\mL^{\infty}(\mathscr{O})$ (see the schematic Figure \ref{FigureConservationNRJ}). This is a consequence of the identity $\Re e\,(\overline{R^+(\rho_0)}dR^+(\rho_0)(\mu))+\Re e\,(\overline{T(\rho_0)}dT(\rho_0)(\mu))=0$ (see Remark \ref{RemarkdNRJ}).

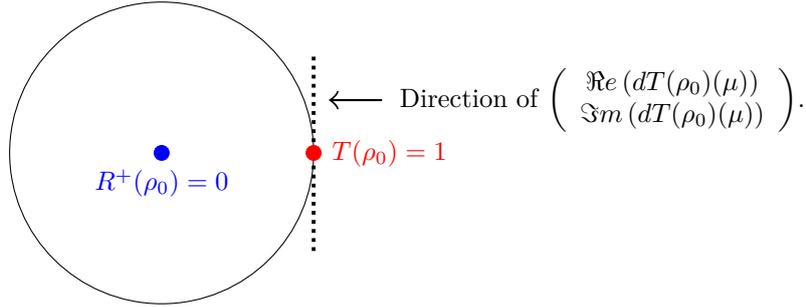
\begin{figure}[!ht]
\centering
\begin{tikzpicture}[scale=1]
\draw (0,0) circle (2);
\draw[fill=blue,draw=none] (0,0) circle (3pt);
\draw[dotted,line width=0.5mm] (2,-1.3)--(2,1.3);
\draw[fill=red,draw=none] (2,0) circle (3pt);
\node at (0,-0.4){\small \textcolor{blue}{$R^+(\rho_0)= 0$}};
\node at (3,0){\small \textcolor{red}{$T(\rho_0)= 1$}};
\node at (5.8,0.7){\small Direction of $\left(\begin{array}{c}\Re e\,(dT(\rho_0)(\mu))\\\Im m\,(dT(\rho_0)(\mu))\end{array}\right)$.};
\scalebox{2}{$\draw[<-] (1.1,0.35)--(1.45,0.35);$}
\end{tikzpicture}
\caption{From the identity of conservation of energy $|R^+|^2+|T|^2=1$, we know that $R^+$ and $T$ must lie in the unit disk of the complex plane. If $\Re e\,(dT(\rho_0)(\mu))$ was not null for some $\mu\in\mL^{\infty}(\Om)$, where $\rho_0$ is such that $T(\rho_0)=1$, then we would have $|T(\rho_0+\eps\mu)|>1$ or $|T(\rho_0-\eps\mu)|>1$ for $\eps>0$ small enough, which is impossible.  \label{FigureConservationNRJ}} 
\end{figure}

\subsection{Relative invisibility}\label{paragraphRelativeInv}

In this paragraph, for a given $\rho_0\in\mL^{\infty}(\mathscr{O})$, we wish to find functions $\rho\not\equiv\rho_0$ such that $\mathbb{S}(\rho)=\mathbb{S}(\rho_0)$. First let us study what can be done with the map $F$ introduced in (\ref{DefFTOne1D}) to impose invisibility in reflection and transmission. We remind the reader that it is defined by $F(\rho_0)=( \Re e\,R^+(\rho_0), \Im m\,R^+(\rho_0),\Im m\,T(\rho_0))$. Proposition \ref{DifferentialT} ensures that $dF(\rho_0):\mrm{L}^{\infty}(\mathscr{O})\to\R^3$ is onto if and only if $\rho_0\in\mrm{L}^{\infty}(\mathscr{O})$ is such that $\Re e\,T(\rho_0)\ne0$.
\begin{proposition}\label{PropoContinuity0}
Let $F$ be as in (\ref{DefFTOne1D}) and $\rho_0\in\mrm{L}^{\infty}(\mathscr{O})$ be such that $T(\rho_0)\ne0$. There exists $\eps>0$ such that for $\|\rho-\rho_0\|_{\mrm{L}^{\infty}(\mathscr{O})} \le \eps$, we have $\mathbb{S}(\rho)=\mathbb{S}(\rho_0)$ if and only if $F(\rho)=F(\rho_0)$. 
\end{proposition}
\begin{proof}
Clearly if $\mathbb{S}(\rho)=\mathbb{S}(\rho_0)$ then $F(\rho)=F(\rho_0)$. Now assume that $F(\rho)=F(\rho_0)$. Then we have $R^+(\rho)=R^+(\rho_0)$ and $\Im m\,T(\rho)=\Im m\,T(\rho_0)$. By conservation of energy, we deduce from $R^+(\rho)=R^+(\rho_0)$ that $|T(\rho)|=|T(\rho_0)|$. Since $T(\rho)$ is close to $T(\rho_0)$ when $\rho$ is close to $\rho_0$, these constraints suffice to guarantee that $T(\rho)=T(\rho_0)$. But the unitarity of $\mathbb{S}(\rho)$ and $\mathbb{S}(\rho_0)$ impose $\overline{R^+(\rho)}T(\rho)+\overline{T(\rho)}R^-(\rho)=0$ and $\overline{R^+(\rho_0)}T(\rho_0)+\overline{T(\rho_0)}R^-(\rho_0)=0$ (see (\ref{Unitary1D})). This implies $R^-(\rho)=R^-(\rho_0)$ when $T(\rho_0)\ne0$.
\end{proof}
\noindent Thus we have a method to impose $\mathbb{S}(\rho)=\mathbb{S}(\rho_0)$ when $\rho_0$ is such that $\Re e\,T(\rho_0)\ne0$. Now we wish to consider the case $\Re e\,T(\rho_0)=0$. Observe that the identity $\mathbb{S}(\rho)=\mathbb{S}(\rho_0)$ is equivalent to have $\overline{\mathbb{S}(\rho_0)}\mathbb{S}(\rho)=\mrm{Id}^{2\times 2}$ because $\mathbb{S}(\rho_0)$ is unitary and symmetric. For a general $\rho\in\mL^{\infty}(\mathscr{O})$, set $\mathbb{M}:=\overline{\mathbb{S}(\rho_0)}\,\mathbb{S}(\rho)$. In order to get $\mathbb{M}=\mrm{Id}^{2\times 2}$, we will impose $\Im m\,\mathbb{M}_{11}=0$, $\mathbb{M}_{21}=0$ and use the fact that $\mathbb{M}$ is unitary. To proceed, we define the map $F:\mrm{L}^{\infty}(\mathscr{O})\to\R^3$ such that 
\begin{equation}\label{UniversalFunctional}
\begin{array}{lcl}
F(\rho)&=&(\Im m\,\mathbb{M}_{11},\Re e\,\mathbb{M}_{21},\Im m\,\mathbb{M}_{21})\\
&=&(\Im m\,(\overline{R^+_0}\,R^++\overline{T_0}\,T),\Re e\,(\overline{T_0}\,R^++\overline{R^-_0}\,T),\Im m\,(\overline{T_0}\,R^++\overline{R^-_0}\,T)),
\end{array}
\end{equation}
where $R^{\pm}_0$, $T_0$ stand for the coefficients of $\mathbb{S}(\rho_0)$.

\begin{remark}\label{RmkParticularValues}
$\bullet$ For $(R^+_0,T_0)=(0,1)$ (invisibility in reflection and transmission), we have $F(\rho)=(\Im m\,T,\Re e\,R^+,\Im m\,R^+)$. This is coherent with the choice we did in \S\ref{paragraphInvRefTrans}.\\
$\bullet$ For $(R^+_0,T_0)=(0,i)$, we have $F(\rho)=(-\Re e\,T,\Im m\,R^+,-\Re e\,R^+)$. This choice is coherent with what we get by derivating the relations of conservation of energy (\ref{Unitary1D}) (see the schematic Figure \ref{FigureConservationNRJUnivseral}).
\end{remark}

\begin{figure}[!ht]
\centering
\begin{tikzpicture}[scale=1]
\draw (0,0) circle (2);
\draw[fill=blue,draw=none] (0,0) circle (3pt);
\draw[dotted,line width=0.5mm] (-1.3,2)--(1.3,2);
\draw[fill=red,draw=none] (0,2) circle (3pt);
\node at (0,-0.4){\small \textcolor{blue}{$R^+(\rho_0)= 0$}};
\node at (0,2.4){\small \textcolor{red}{$T(\rho_0)= i$}};
\node at (5.8,0.7){\small Direction of $\left(\begin{array}{c}\Re e\,(dT(\rho_0)(\mu))\\\Im m\,(dT(\rho_0)(\mu))\end{array}\right)$.};
\scalebox{2}{$\draw[<-] (0.85,1)--(1.1,1)--(1.1,0.35)--(1.45,0.35);$}
\end{tikzpicture}
\caption{When $\rho_0$ is such that $T(\rho_0)=i$, if $\Im m\,(dT(\rho_0)(\mu))$ was not null for some $\mu\in\mL^{\infty}(\Om)$, then we would have $|T(\rho_0+\eps\mu)|>1$ or $|T(\rho_0-\eps\mu)|>1$ for $\eps>0$ small enough, which is impossible. Therefore, in this case it is natural to 
work with $F(\rho)=(\Re e\,R^+,\Im m\,R^+,\Re e\,T)$.
 \label{FigureConservationNRJUnivseral}} 
\end{figure}
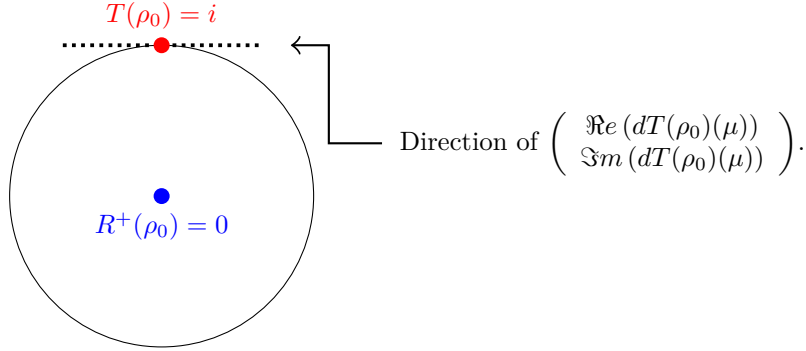

\begin{proposition}\label{PropoContinuity}
Let $F$ be as in (\ref{UniversalFunctional}). There exists $\eps>0$ such that for $\|\rho-\rho_0\|_{\mrm{L}^{\infty}(\mathscr{O})} \le \eps$, we have $\mathbb{S}(\rho)=\mathbb{S}(\rho_0)$ if and only if $F(\rho)=0$. 
\end{proposition}
\begin{proof}
If $\mathbb{S}(\rho)=\mathbb{S}(\rho_0)$ then, since $\mathbb{S}(\rho_0)$ is unitary we have $F(\rho)=0$.\\
Conversely, assume that $F(\rho)=0$. In that case, we find that $\mathbb{M}=\overline{\mathbb{S}(\rho_0)}\,\mathbb{S}(\rho)$ is such that $\mathbb{M}_{11}$ is real and $\mathbb{M}_{21}=0$. Since the product of two unitary matrices is unitary, we know that $\mathbb{M}$ is unitary. Therefore, we must have $|\mathbb{M}_{11}|=|\mathbb{M}_{22}|=1$ and $\mathbb{M}_{12}=0$. From the fact that $\mathbb{M}_{11}$ is real and $\|\rho-\rho_0\|_{\mrm{L}^{\infty}(\mathscr{O})} \le \eps$, we infer that $\mathbb{M}_{11}=1$. Thus, we have
\[
\overline{\mathbb{S}(\rho_0)}\,\mathbb{S}(\rho)=\left(\begin{array}{cc}
1 & 0\\
0 & e^{i\eta}
\end{array}\right)
\]
for some $\eta\in[0;2\pi)$. Multiplying the above equality on the left by $\mathbb{S}(\rho_0)$, we find
\[
\mathbb{S}(\rho)=\left(\begin{array}{cc}
R^+_0 & T_0\,e^{i\eta}\\
T_0 & R^-_0\,e^{i\eta}
\end{array}\right).
\]
Finally, using that $\mathbb{S}(\rho)$ is symmetric, we obtain $\eta=0$ and so $\mathbb{S}(\rho)=\mathbb{S}(\rho_0)$.
\end{proof}

\begin{proposition}\label{DifferentialUniversal}
Set $0<k<\pi$ ($N=1$). Then the map $dF(\rho_0):\mrm{L}^{\infty}(\mathscr{O})\to\R^3$ with $F$ defined in (\ref{UniversalFunctional}) is onto if and only if $\rho_0\in\mrm{L}^{\infty}(\mathscr{O})$ is such that $T_0=T(\rho_0)\ne0$.
\end{proposition}
\begin{proof}
Using identities (\ref{DiffSca1D}), we obtain 
\[
dF(\rho_0)(\mu)=k^2(\dsp\int_{\Om}\mu |u^+|^2\,dz,-\Im m\,\dsp\int_{\Om}\mu u^+\,\overline{u^-}\,dz,\Re e\,\dsp\int_{\Om}\mu u^+\,\overline{u^-}\,dz).
\]
We emphasize that here the $u^{\pm}$ are the total fields for the problem (\ref{PbChampTotalBIS}) with physical coefficient $\rho_0$. From Lemma \ref{lemmaGram}, we deduce that $dF(\rho_0):\mrm{L}^{\infty}(\mathscr{O})\to\R^3$ is onto if and only if the family of real functions $\{|u^+|^2,\Im m\,(u^+\,\overline{u^-}), \Re e\,(u^+\,\overline{u^-})\}$ is linearly independent.\\ [4pt]
$\star$ First, we consider the situation where $T_0\ne0$. From formula (\ref{ImportantRelation1D}), we know that $\overline{u^-}=(u^+-R^+_0\overline{u^+})/T_0$. We deduce
\begin{equation}\label{eqnOnto1Univ}
u^+\overline{u^-} =  (u^+)^2/T_0-R^+_0|u^+|^2/T_0.
\end{equation}
Introduce the real constants $A$, $B$, $C$, $D$ such that $1/T_0=A+iB$ and $-R^+_0/T_0=C+iD$. Note that there holds $(A,B)\ne(0,0)$. Using these notations in (\ref{eqnOnto1Univ}), we get
\[
\begin{array}{ll}
\Re e\,((u^+)^2/T_0)=A\,\Re e\,((u^+)^2)-B\,\Im m\,((u^+)^2);&\ \Im m\,((u^+)^2/T_0)=B\,\Re e\,((u^+)^2)+A\,\Im m\,((u^+)^2);\\[6pt]
\Re e\,(-R^+_0|u^+|^2/T_0)=C\,|u^+|^2; &\ \Im m\,(-R^+_0|u^+|^2/T_0)=D\,|u^+|^2;
\end{array}
\]
and so 
\begin{equation}\label{ExpansionReIm}
\begin{array}{rcl}
\Im m\,(u^+\overline{u^-})&\hspace{-0.15cm}=&\hspace{-0.15cm}D\,|u^+|^2+B\,\Re e\,((u^+)^2)+A\,\Im m\,((u^+)^2)\\[2pt]
\Re e\,(u^+\overline{u^-})&\hspace{-0.15cm}=&\hspace{-0.15cm}C\,|u^+|^2+A\,\Re e\,((u^+)^2)-B\,\Im m\,((u^+)^2).
\end{array}
\end{equation}
Assume that there are some $\alpha$, $\beta$, $\gamma\in\R$ such that 
\begin{equation}\label{eqnOnto2Univ}
\alpha\,|u^+|^2+\beta\,\Im m\,(u^+\overline{u^-})+\gamma\,\Re e\,(u^+\overline{u^-})=0\quad \mbox{ in }\mathscr{O}.
\end{equation}
Inserting (\ref{ExpansionReIm}) in (\ref{eqnOnto2Univ}), we find that we must have
\begin{equation}\label{eqnOnto2UnivEquiv}
\eta_1\,|u^+|^2+\eta_2\,\Re e\,((u^+)^2)+\eta_3\,\Im m\,((u^+)^2)=0\quad \mbox{ in }\mathscr{O}\qquad\mbox{ with }\quad\begin{array}{|l}
\eta_1=\alpha+\beta\,D+\gamma\,C\\
\eta_2=\beta\,B+\gamma\,A\\
\eta_3=\beta\,A-\gamma\,B.
\end{array}
\end{equation}
Set again $a=\Re e\,u^+$ and $b=\Im m\,u^+$. Equation (\ref{eqnOnto2UnivEquiv}) is equivalent to 
\[
(\eta_1+\eta_2)\,a^2+\eta_3\,2ab+(\eta_1-\eta_2)\,b^2=0\quad \mbox{ in }\mathscr{O}.
\]
Working as in the proof of Proposition \ref{DifferentialR}, we obtain $\eta_1+\eta_2=\eta_3=\eta_1-\eta_2=0$ and so $\eta_1=\eta_2=\eta_3=0$. From the equations $\eta_2=\eta_3=0$ and the fact that $(A,B)\ne(0,0)$, we deduce that $\beta=\gamma=0$. Then, since $\eta_1=0$, we must also have $\alpha=0$. Thus $\{|u^+|^2,\Im m\,(u^+\,\overline{u^-}), \Re e\,(u^+\,\overline{u^-})\}$ is a family of linearly independent functions.\\[4pt]
$\star$ Now we assume that $T_0=0$. Set
\[
a=\Re e\,u^+;\qquad b=\Im m\,u^+;\qquad c=\Re e\,u^-;\qquad d=\Im m\,u^-.
\]
From Proposition \ref{propositionImagMono}, we know that there are some real constants $A$, $B$, $C$, $D$ with $(A,B)\ne(0,0)$ and $(C,D)\ne(0,0)$ such that
\[
A\,a+B\,b=0\qquad\mbox{ and }\qquad C\,c+D\,d=0\quad\mbox{ in }\mathscr{O}.
\]
Then one can check that there are $\beta$, $\gamma\in\R$ with $(\beta,\gamma)\ne(0,0)$ such that 
\begin{equation}\label{eqnOnto2UnivBis}
\begin{array}{ll}
&\beta\,\Im m\,(u^+\overline{u^-})+\gamma\,\Re e\,(u^+\overline{u^-})=0\quad \mbox{ in }\mathscr{O}\\[6pt]
\Leftrightarrow & \beta\,(bc-ad)+\gamma\,(ac+bd)=0\quad\mbox{ in }\mathscr{O}\\[6pt]
\Leftrightarrow & a(\gamma c-\beta d)+b(\beta c+\gamma d)=0\quad\mbox{ in }\mathscr{O}.
\end{array}
\end{equation}
Indeed if $a\equiv0$, then one can take $\beta=C$ and $\gamma=D$. If $b\equiv0$, then one can take $\beta=-D$ and $\gamma=C$. And if both $a\not\equiv0$ and $b\not\equiv0$, one has 
\[
a(\gamma c-\beta d)+b(\beta c+\gamma d)=0\quad\mbox{ in }\mathscr{O}\quad\Leftrightarrow\quad (\gamma-\beta A/B) c- (\beta +\gamma A/B)d=0.
\]
Thus one has to solve the system $\gamma-\beta A/B=C$ and $- (\beta +\gamma A/B)=D$ with respect to $(\beta,\gamma)$. And this system admits a non zero solution. 
\end{proof}

\begin{remark}\label{RmkDiffNRJBis}
In the Remark \ref{RmkDiffNRJ}, we saw that when $T(\rho_0)=0$, the map    $dR(\rho_0):\mL^{\infty}(\mathscr{O})\to\Cplx$ is not onto. Let us prove that in this case $dT(\rho_0):\mL^{\infty}(\mathscr{O})\to\Cplx$ is not onto either. The unitarity of $\mathbb{S}$ imposes $\overline{R^+}T+\overline{T}R^-=0$ (see (\ref{Unitary1D})). Differentiating this identity and using that $T(\rho_0)=0$, we find 
$\overline{R^+_0}dT(\rho_0)(\mu)+R^-_0\overline{dT(\rho_0)(\mu)}=0$. As a consequence, there are some real constants $A$, $B$ independent of $\mu\in\mL^{\infty}(\mathscr{O})$, with $(A,B)\ne(0,0)$, such that $A\,\Re e\,(dT(\rho_0)(\mu))+B\,\Im m\,(dT(\rho_0)(\mu))=0$. Thus  $dT(\rho_0):\mL^{\infty}(\mathscr{O})\to\Cplx$ cannot be onto.
\end{remark}

\noindent In the remaining part of this paragraph, we explain how to impose $\mathbb{S}(\rho)=\mathbb{S}(\rho_0)$ when $\mathbb{S}(\rho_0)$ is such that $T_0=T(\rho_0)=0$. To proceed, we need to work with a new functional $F$ because the one of (\ref{UniversalFunctional}) is such that $dF(\rho_0):\mrm{L}^{\infty}(\mathscr{O})\to\R^3$ is not onto (Proposition \ref{DifferentialUniversal}) in this case. Differentiating the relations (\ref{Unitary1D}) and using that $T(\rho_0)=0$, we obtain, for all $\mu\in\mL^{\infty}(\mathscr{O})$,
\[
\Re e\,(\overline{R^+_0}dR^+(\rho_0)(\mu))=0;\qquad\Re e\,(\overline{R^-_0}dR^-(\rho_0)(\mu))=0;\qquad \Re e\,(\sqrt{\overline{R^+_0}\overline{R^-_0}}\,dT(\rho_0)(\mu))=0
\]
(for the third one, from (\ref{Unitary1D}) we get $\overline{R^+_0}dT(\rho_0)(\mu)+R^-_0d\overline{T}(\rho_0)(\mu)=0$ and then we use that $|R^+_0|=|R^-_0|=1$). The first and third relations have already been obtained in Remarks \ref{RmkDiffNRJ} and \ref{RmkDiffNRJBis} respectively. This leads us to define the map $F:\mrm{L}^{\infty}(\mathscr{O})\to\R^3$ such that 
\begin{equation}\label{UniversalFunctionalTNull}
F(\rho)=(\Im m\,(\overline{R^+_0}\,R^+),\Im m\,(\overline{R^-_0}\,R^-),\Im m\,(\sqrt{\overline{R^+_0}\overline{R^-_0}}\,T)).
\end{equation}
\begin{proposition}\label{PropoContinuity2}
Let $F$ be as in (\ref{UniversalFunctionalTNull}) and $\rho_0\in\mrm{L}^{\infty}(\mathscr{O})$ be such that $T_0=T(\rho_0)=0$. There exists $\eps>0$ such that for $\|\rho-\rho_0\|_{\mrm{L}^{\infty}(\mathscr{O})} \le \eps$, we have $\mathbb{S}(\rho)=\mathbb{S}(\rho_0)$ if and only if $F(\rho)=0$. 
\end{proposition}
\begin{proof}
Let $\rho_0\in\mrm{L}^{\infty}(\mathscr{O})$ be such that $T_0=T(\rho_0)=0$. First observe that if $\rho$ is such that
$\mathbb{S}(\rho)=\mathbb{S}(\rho_0)$, then we have $R^{\pm}=R^{\pm}_0$ and $T=T_0=0$. This implies $\Im m\,(\overline{R^{\pm}_0}\,R^{\pm})=\Im m\,(|R^{\pm}_0|^2)=0$ and so $F(\rho)=F(\rho_0)=0$.\\
Now assume that $\rho$ is such that $\|\rho-\rho_0\|_{\mrm{L}^{\infty}(\mathscr{O})} \le \eps$ for $\eps$ small enough and $F(\rho)=0$. When $T_0=0$, we have $|R^+_0|=|R^-_0|=1$ and so there are $\alpha^{\pm}\in[0;2\pi)$ such that $R^{\pm}_0=e^{i\alpha^{\pm}_0}$. On the other hand, we can write the coefficients $R^{\pm}$ as $R^{\pm}=|R^{\pm}|\,e^{i\alpha^{\pm}}$ with $\alpha^{\pm}\in[0;2\pi)$. Since $R^{\pm}$ are small perturbations of $R^{\pm}_0$ for $\eps$ small, we deduce that $|R^{\pm}|>0$. When $F(\rho)=0$, there holds $\Im m\,(\overline{R^+_0}\,R^+)=\Im m\,(\overline{R^-_0}\,R^-)=0$ and so $\alpha^+=\alpha^+_0$, $\alpha^-=\alpha^-_0$ (again here we use the argument of small perturbation). When $F(\rho)=0$, we also have
\begin{equation}\label{relPhase1}
0=\Im m\,(\sqrt{\overline{R^+_0}\overline{R^-_0}}\,T).
\end{equation}
But the unitarity of $\mathbb{S}(\rho)$ implies $0=\overline{R^+}T+\overline{T}R^-$ and $|R^{+}|=|R^{-}|$. Dividing by $|R^{\pm}|$, we get $0=\overline{R^+_0}T+\overline{T}R^-_0$ which leads to
\begin{equation}\label{relPhase2}
0=\Re e\,(\sqrt{\overline{R^+_0}\overline{R^-_0}}\,T).
\end{equation}
From (\ref{relPhase1}), (\ref{relPhase2}), we infer that $T=0$ and so $|R^{\pm}|=1$ by conservation of energy. This gives $R^+=R^+_0$, $R^-=R^-_0$ and so $\mathbb{S}(\rho)=\mathbb{S}(\rho_0)$. 
\end{proof}
\noindent Now we study the question of the ontoness of the differential of the function $F$ in (\ref{UniversalFunctionalTNull}). 
\begin{proposition}\label{DifferentialTNull}
Set $0<k<\pi$ ($N=1$) and assume that $T_0=T(\rho_0)=0$. Then the map $dF(\rho_0):\mrm{L}^{\infty}(\mathscr{O})\to\R^3$ with $F$ defined in (\ref{UniversalFunctionalTNull}) is onto.
\end{proposition}
\begin{proof}
When $T_0=0$, according to (\ref{ImportantRelation1D}), we have
\begin{equation}\label{ConstitutiveTNull}
\overline{R^+_0}u^+=\overline{u^+}\qquad\mbox{ and }\qquad\overline{R^-_0}u^-=\overline{u^-}.
\end{equation}
Denote $\kappa=\sqrt{\overline{R^+_0R^-_0}}$. Since $|\kappa|=1$ so that $\kappa^{-1}=\overline{\kappa}$, from (\ref{ConstitutiveTNull}), we infer that 
\begin{equation}\label{RelationCplx}
\kappa\,u^+u^-=\overline{\kappa\,u^+u^-},
\end{equation}
which ensures that $\kappa\,u^+u^-$ is real. Using (\ref{DiffSca1D}) and (\ref{ConstitutiveTNull}), we get
\[
dF(\rho_0)(\mu)=k^2(\dsp\int_{\Om}\mu |u^+|^2\,dz,\dsp\int_{\Om}\mu |u^-|^2\,dz,\dsp\int_{\Om}\mu\,\kappa\,u^+u^-\,dz).
\]
Note that to obtain the third component of $dF(\rho_0)(\mu)$, from (\ref{DiffSca1D}), we wrote successively
\[
\Im m\,(\sqrt{\overline{R^+_0}\overline{R^-_0}}\,dT(\rho_0)(\mu))=k^2\dsp\int_{\Om}\mu\,\Re e\,(\sqrt{\overline{R^+_0}\overline{R^-_0}}u^+u^-)\,dz=k^2\dsp\int_{\Om}\mu\,\kappa\,u^+u^-\,dz.
\]
From Lemma \ref{lemmaGram}, we infer that $dF(\rho_0):\mrm{L}^{\infty}(\mathscr{O})\to\R^3$ is onto if and only if $\{|u^+|^2,|u^-|^2,\kappa\,u^+u^-\}$ is a family of linearly independent functions. Set $a=\Re e\,u^+$, $b=\Im m\,u^+$, $c=\Re e\,u^-$ and $d=\Im m\,u^-$. From (\ref{ConstitutiveTNull}), we know that there are some real constants $A$, $B$, $C$, $D$ with $(A,B)\ne(0,0)$ and $(C,D)\ne(0,0)$ such that
\begin{equation}\label{ABCDeqn}
A\,a+B\,b=0\qquad\mbox{ and }\qquad C\,c+D\,d=0\quad\mbox{ in }\Om.
\end{equation}
Assume that are some $\alpha$, $\beta$, $\gamma\in\R$ such that 
\begin{equation}\label{RelatSign}
\alpha\,|u^+|^2+\beta\,|u^-|^2+\gamma\,\kappa\,u^+u^-=0\quad \mbox{ in }\mathscr{O}.
\end{equation}
Since $u^+$ and $u^-$ cannot be null on non empty open sets, we know that there is a non empty open set $\mathscr{O}'\subset\mathscr{O}$ where $\kappa\,u^+u^-$ does not vanish. Then in $\mathscr{O}'$, we have $\kappa\,u^+u^-=\pm |u^+|\,|u^-|$ (remember that $|\kappa|=1$). Assume first that $ABCD\ne0$. Then from (\ref{RelatSign}), we get 
\[
\alpha(1+A^2/B^2)\,a^2+\beta(1+C^2/D^2)\,c^2\pm\gamma\sqrt{1+A^2/B^2}\sqrt{1+C^2/D^2}\,|a|\,|c|=0\quad \mbox{ in }\mathscr{O}'.
\]
According to Lemma \ref{lem-Q(X,Y)} (observing that $a$, $c$ are continuous, one can verify that one can use this lemma), if $(\alpha,\beta,\gamma)\ne(0,0,0)$, then there are $(\lambda_1,\lambda_2)\ne(0,0)$ such that $\lambda_1\,a+\lambda_2\,c=0$. This is impossible because $a$ (resp. $c$) is exponentially decaying as $x\to+\infty$ (resp. $x\to-\infty$) while $c$ (resp. $a$) is not. Thus, we must have $\alpha= \beta=\gamma=0$. \\
The different cases where $ABCD=0$ can be dealt with in a similar way. 
\end{proof}

\noindent Finally, to impose relative invisibility in monomode regime, the situation is as follows. When $\rho_0$ is such that $\Re e\,T_0\ne0$, one can work with the functional  $F(\rho)=( \Re e\,R^+, \Im m\,R^+,\Im m\,T)$ defined in (\ref{DefFTOne1D}). When $T_0\ne0$, one can work with the functional $F$ defined in (\ref{UniversalFunctional}). When $T_0=0$, one can work with the functional $F$ defined in (\ref{UniversalFunctionalTNull}). And from Theorem \ref{the-ptfixe} as well as Propositions \ref{PropoDiffScaMat}, \ref{PropoContinuity0}, \ref{PropoContinuity}, \ref{DifferentialUniversal}, \ref{PropoContinuity2}, \ref{DifferentialTNull}, we can state the following result. Here $\mathscr{\K}=\mrm{span}(\mu_1,\mu_2,\mu_3)$ is a subspace of $\mrm{L}^{\infty}(\mathscr{O})$ of dimension $3$ such that $dF(\rho_0):\mathscr{\K}\to\R^3$ is a bijection. 
\begin{theorem}\label{RelativeInv1D}
Set $0<k<\pi$ ($N=1$). Assume that $\rho_0\in\mrm{L}^{\infty}(\mathscr{O})$ is such that trapped modes do not exist for the problem (\ref{InitialPb}). Let $\mu_0$ be a non-trivial element of $dF(\rho_0)$. Then for all $\radius>0$, there is $\eps_0>0$ such that for all $\eps\in(0;\varepsilon_0]$
\[
\exists!\tau=(\tau_1,\tau_2,\tau_3)\in \overline{B(O,\radius)}\subset\R^3\mbox{ such that }\mathbb{S}(\rho_0+\varepsilon(\mu_0+\sum_{j=1}^3\tau_j\mu_j))=\mathbb{S}(\rho_0).
\]
\end{theorem}
\begin{remark}
Note that the statement of Theorem \ref{PerfectTrans1D} is contained in the result of Theorem \ref{RelativeInv1D}. It corresponds to the case where $\rho_0$ is such that $\mathbb{S}(\rho_0)=\mrm{Id}^{2\times2}$.
\end{remark}

\noindent The results of Theorems \ref{NonRef1D}, \ref{PerfectTrans1D} and \ref{RelativeInv1D} are still rather abstract. We will show in the next section how to choose the $\mu_0$, $\mu_i$ to construct non reflecting, perfectly invisible or relatively invisible obstacles.\\
Before proceeding further, let us explain why, unfortunately, we did not succeed in proving similar results of ontoness of $dF(\rho_0)$ in multimode regime. Due to the general expression (\ref{ExpressionDifferentialsMulti}) of the differentials of the scattering coefficients, similar proofs in multimode regime would require to generalize Lemma \ref{lem-Q(X,Y)} to the case where the quadratic form $Q$ has more than two arguments. But we conjecture that such a generalization does not hold.\\
However, for the invisibility in reflection as well as for the invisibility in reflection and transmission, we have proposed in \S\ref{paragraphProblems} a choice of functionals $F$ in multimode regime such that the most obvious reasons of non-ontoness of the differentials are eliminated. We will see that satisfactory results are indeed  obtained numerically with such $F$. For relative invisibility in multimode regime, the choice for $F$ still needs to be studied.

\section{Numerical examples}\label{SectionNumerics}

\subsection{General procedure}\label{paragraphGenProc}
Before presenting the numerical results, we explain how we obtain them. Let $\mathscr{O}\subset\Om=\R\times(0;1)$ be a non empty bounded open set which is given once for all. We assume that $\overline{\mathscr{O}}\subset(-\ell;\ell)\times(0;1)$ with $\ell=5$. We work with functional $F(\rho)=(F_i(\rho))_{i=1}^d$ defined as in the previous sections and valued in $\R^d$, $d\ge1$. For a given $\rho_0\in\mL^{\infty}(\mathscr{O})$, we want to construct $\rho\in\mL^{\infty}(\mathscr{O})$ with $\rho\not\equiv\rho_0$ such that $F(\rho)=F(\rho_0)$. We look for $\rho$ of the form $\rho=\rho_0+\eps\mu$ with $\eps$ small and with $\mu$ such that
\begin{equation}\label{DefMuGene}
\mu=\mu_0+\sum_{j=1}^d\tau_j\,\mu_j.
\end{equation}
Here the $\tau_j$ are real numbers to compute and the $\mu_j\in\mL^{\infty}(\mathscr{O})$ are such that
\begin{equation}\label{BasisFunctionsGene}
dF_i(\rho_0)(\mu_j)=\delta_{ij}.
\end{equation}
Again, we emphasize that the $\mu_j$, $j=0,\dots,d$, are well-defined when $dF(\rho_0):\mL^{\infty}(\mathscr{O})\to\R^d$ is onto. However clearly they are not uniquely defined. Let us explain how we choose them in the numerical procedure. According to the results of \S\ref{ParagraphDiff} and in particular (\ref{ExpressionDifferentialsMulti}), for the $F$ considered in the previous sections, we have for all $\mu\in\mrm{L}^{\infty}(\mathscr{O})$,
\begin{equation}\label{diffFGene}
dF_i(\rho_0)(\mu)=\int_{\Om}\mu f_i\,dz
\end{equation}
for certain functions $f_i\in\mathscr{C}^0(\overline{\mathscr{O}})$ (involving the $u_m^{\pm}$ introduced in (\ref{scatteredFieldParticular})). Define the Gram matrix
\[
\G:=\left(\int_{\mathscr{O}} f_if_j\,dz \right)_{1 \le i,j\le d}.
\]
Denote $\mathbb{H}=(\mathbb{H}_{ij})_{1 \le i,j\le d}$ the inverse of $\G$ which is well-defined when $dF(\rho_0):\mL^{\infty}(\mathscr{O})\to\R^d$ is onto (this is what we used in the proof of Lemma \ref{lemmaGram}). Finally, set
\[
\mu_j=\sum_{i=1}^d\mathbb{H}_{ji}f_i.
\]
Then from a $\mu_0^{\#}$ such that $\mu_0^{\#}\notin \mrm{span}(\mu_1,\dots,\mu_d)$, we define $\mu_0$  by 
\[
\mu_0:=\mu_0^{\#}-\sum_{j=1}^ddF_j(\rho_0)(\mu_0^{\#})\,\mu_j.
\]
One can verify that with such definitions, the functions $\mu_0$, $\mu_1,\dots,\mu_d$ satisfy (\ref{BasisFunctionsGene}).
\begin{remark}
For the simplest problem of invisibility in reflection in monomode regime, that is when $0<k<\pi$ ($N=1$), we saw in \S\ref{ParagRNullMono} that we can take $F(\rho)=(\Re e\,R^+(\rho),\Im m\,R^+(\rho))\in\R^2$. Then from (\ref{DiffSca1D}), we know that for all $\mu\in\mrm{L}^{\infty}(\mathscr{O})$, we have
\[
dR^+(\rho)(\mu)=ik^2\dsp\int_{\Om}\mu (u^+)^2\,dz.
\]
In this situation, the $f_i$ introduced in (\ref{diffFGene}) are simply given by $f_1=\Re e\,(ik^2(u^+)^2$) and $f_2=\Im m\,(ik^2(u^+)^2)$. 
\end{remark}
\noindent For $\mu$ as in (\ref{DefMuGene}), we have the expansion 
\[
\begin{array}{lcl}
F(\rho_0+\eps\mu )&=&F(\rho_0)+\eps dF(\rho_0)(\mu)+\eps^2\widetilde{F}^{\eps}(\tau)\\[5pt]
&=&F(\rho_0)+\eps \tau+\eps^2\widetilde{F}^{\eps}(\tau),
\end{array}
\]
where $\widetilde{F}^{\eps}(\tau)\in\R^d$ is an abstract remainder and where $\tau:=(\tau_1,\dots,\tau_d)\in\R^d$. Thus, to impose $F(\rho_0+\eps\mu )=F(\rho_0)$, we see that $\tau$ must verify the fixed point equation
\begin{equation}\label{FixedPointEqnGene}
\begin{array}{|l}
\mbox{Find }\tau\in\R^d\\
\tau=\mathscr{H}^{\eps}(\tau)
\end{array}\qquad\mbox{ with }\qquad  \mathscr{H}^{\eps}(\tau):=-\eps\widetilde{F}^{\eps}(\tau).
\end{equation}
Numerically, we solve (\ref{FixedPointEqnGene}) using an iterative procedure. We start from $\tau^{0}=0$ and for all $p\in\N$, we set $\tau^{p+1}=\mathscr{H}^{\eps}(\tau^{p})$. Using the definition of $\mathscr{H}^{\eps}$, one observes that there holds $\mathscr{H}^{\eps}(\tau^{p})=\tau^{p}+\eps^{-1}(F(\rho_0)-F(\rho_0+\eps\mu^p))$ with $\mu^p$ defined as in (\ref{DefMuGene}) with $\tau$ replaced by $\tau^{p}$. Therefore, we have
\begin{equation}\label{IterativeProc}
\tau^{p+1}=\tau^{p}+\eps^{-1}(F(\rho_0)-F(\rho_0+\eps\mu^p)).
\end{equation}
We stop the loop when we have $|\tau^{p+1}-\tau^{p}| \le \eta$ where $\eta>0$ is a small given criterion. We then define $\tau^{\mrm{sol}}$ as the last value of $\tau^{p}$. Then we have $F(\rho_0+\eps\mu^{\mrm{sol}})\approx F(\rho_0)$. If the iterative process does not converge, we try again with a smaller value of $\eps>0$. Note that at each step $j\ge0$ of the procedure, we need to solve scattering problems of the form
\begin{equation}\label{PbNum}
\begin{array}{|rl}
\multicolumn{2}{|l}{\mbox{Find }u\in\mrm{H}^1_{\loc}(\Om) \mbox{ such that }u-u_i\mbox{ is outgoing and } }\\[3pt]
\Delta u + k^2(1+\rho_0+\eps\mu^p) u = 0 & \mbox{ in }\Om \\[3pt]
\partial_{y} u=0 & \mbox{ on }\partial\Om.
\end{array}
\end{equation}
To proceed, we approximate the solution of (\ref{PbNum}) with a P2 finite element method in $\Om_5:=\{(x,y)\in\Om\,|\,|x|<5\}$. At $x=\pm 5$, a truncated Dirichlet-to-Neumann map with 10 terms serves as a transparent boundary condition. Computations are implemented with \textit{FreeFem++}\footnote{\textit{FreeFem++}, \url{http://www.freefem.org/ff++/}.} as well as with \textit{XLiFE++}\footnote{\textit{XLiFE++}, \url{https://uma.ensta-paris.fr/soft/XLiFE++/}.} while results are displayed with \textit{Paraview}\footnote{\textit{Paraview}, \url{http://www.paraview.org/}.}.\\
\newline
Once $\rho_1=\rho_0+\eps\mu^{\mrm{sol}}$ has been constructed such that $F(\rho_1)=F(\rho_0)$, one can perturb it to construct another $\rho_2\in\mL^{\infty}(\mathscr{O})$ such that $F(\rho_2)=F(\rho_1)=F(\rho_0)$. We denote by $\aleph$ the number of times we repeat the perturbative construction.
\begin{remark}
One can remark that the method also works when $\mu_0$ in the decomposition (\ref{DefMuGene}) is not chosen in $\ker\,dF(\rho_0)$. In this case, we have 
\[
F(\rho_0+\eps\mu )=F(\rho_0)\qquad\Leftrightarrow\qquad \tau= \tilde{\mathscr{H}}^{\eps}(\tau)
\]
with $\tilde{\mathscr{H}}^{\eps}(\tau):=-dF(\rho_0)(\mu_0)-\eps\widetilde{F}^{\eps}(\tau)$. One can verify that the Banach fixed point theorem guarantees that the above fixed point problem admits a unique solution. Numerically, one can check that it leads to solve exactly the same iterative problem as in \eqref{IterativeProc}.
\end{remark}

\subsection{Results}

\noindent\textit{i)} First, in Figures \ref{rhoSimu1}--\ref{rhoSimu1FieldsR} we impose invisibility in reflection. We work with $k=0.8\pi\in(0;\pi)$ so that only one mode can propagate. In this case, we just have to cancel one complex reflection coefficient, this is the setting of Theorem \ref{NonRef1D}. The obstacle has a rectangular shape. In Figure \ref{rhoSimu1}, we display the sequence of non reflecting obstacles. Here we repeated three times the perturbative construction ($\aleph=3$). Note that the amplitude of the perturbation increases at each step. And we could have continued the process working with a larger $\aleph$. Each of the non reflecting $\rho$ has been obtained by solving the fixed point problem (\ref{FixedPointEqnGene}) via the iterative procedure (\ref{IterativeProc}). The Figure \ref{rhoSimu1FieldsR} represents the real part of the total and scattered fields for the last $\rho$ of Figure \ref{rhoSimu1}. As expected, we observe that the scattered field is exponentially decaying as $x\to-\infty$. And there is a shift of phase in the transmission. This is normal because we cancel only the reflection coefficient.

\begin{figure}[!ht]
\centering
\includegraphics[width=5.2cm]{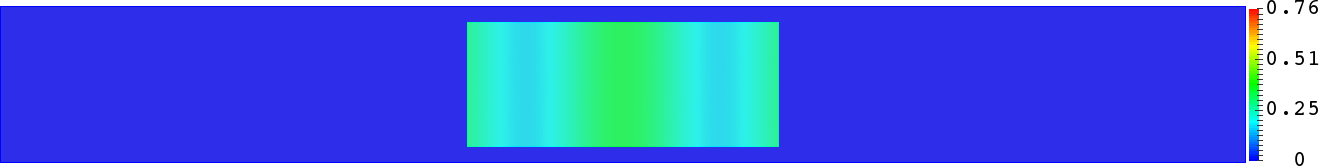}\quad \includegraphics[width=5.2cm]{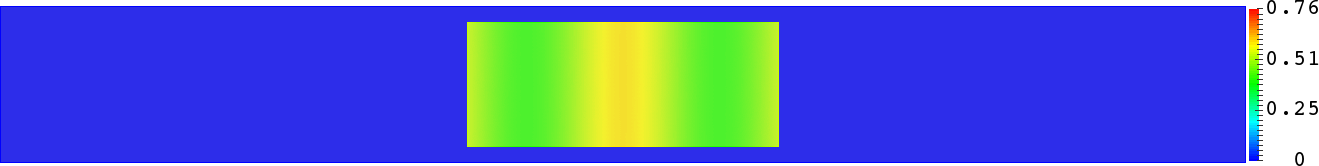}\quad \includegraphics[width=5.2cm]{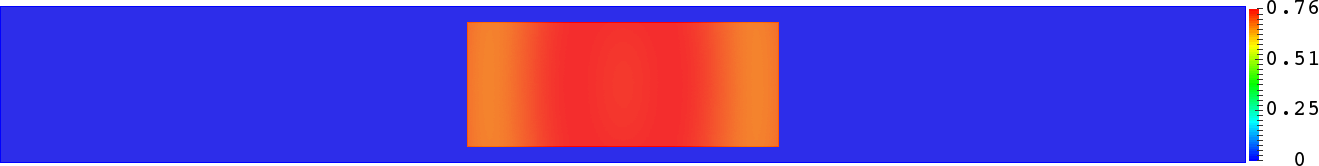}
\caption{Evolution of the parameter $\rho$. We emphasize that for each image, the represented $\rho$ is non reflecting and has been obtained solving the fixed point problem (\ref{FixedPointEqnGene}). \label{rhoSimu1}}
\end{figure}

\begin{figure}[!ht]
\centering
\includegraphics[width=8cm]{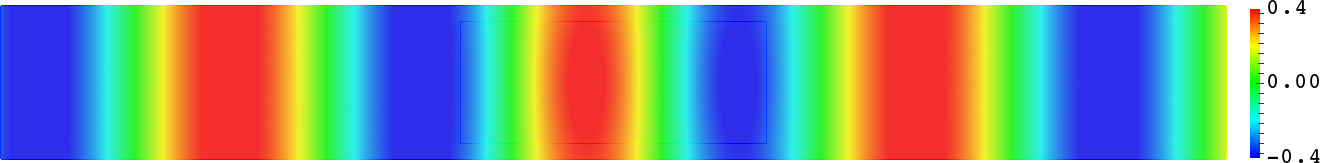}\quad\includegraphics[width=8cm]{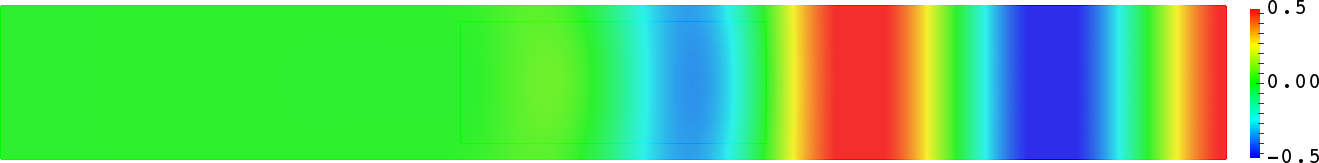}
\caption{Real parts of the total field $u^+$ (left) and of the scattered field $u^+-w^+$ (right) for the last $\rho$ of Figure \ref{rhoSimu1}.\label{rhoSimu1FieldsR}}
\end{figure}

\noindent\textit{ii)} In Figures \ref{rhoSimu1T}--\ref{rhoSimu1FieldsT}, we impose invisibility in reflection and transmission. We work with $k=0.8\pi\in(0;\pi)$ so that the setting is the one of Theorem \ref{PerfectTrans1D}. Following the statement of this theorem, we cancel $R^+$ as well as $\Im m\,T$. The obstacle has the same rectangular shape as in the previous series of experiments. In Figure \ref{rhoSimu1T}, we display the sequence of perfectly invisible obstacles (again we take $\aleph=3$ to set ideas). And the Figure \ref{rhoSimu1FieldsT} represents the real part of the total and scattered fields for the last $\rho$ of Figure \ref{rhoSimu1T}. As desired, this time the scattered field is exponentially decaying both at minus and plus infinity.

\begin{figure}[!ht]
\centering
\includegraphics[width=5.2cm]{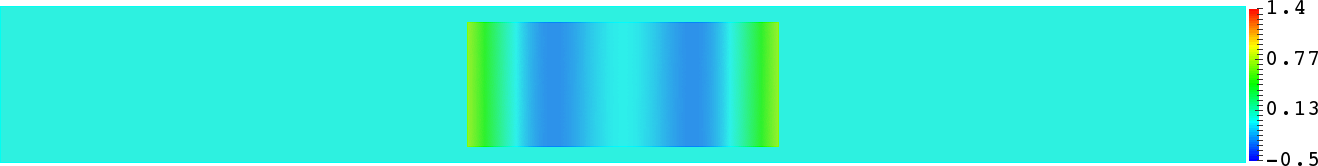}\quad\includegraphics[width=5.2cm]{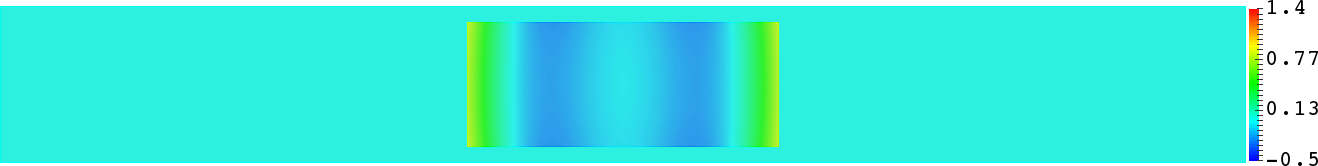}\quad\includegraphics[width=5.2cm]{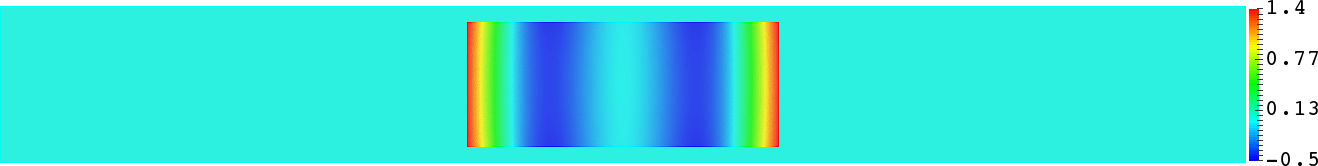}
\caption{Evolution of the perfectly invisible parameter $\rho$. \label{rhoSimu1T}}
\end{figure}
\begin{figure}[!ht]
\centering
\includegraphics[width=8cm]{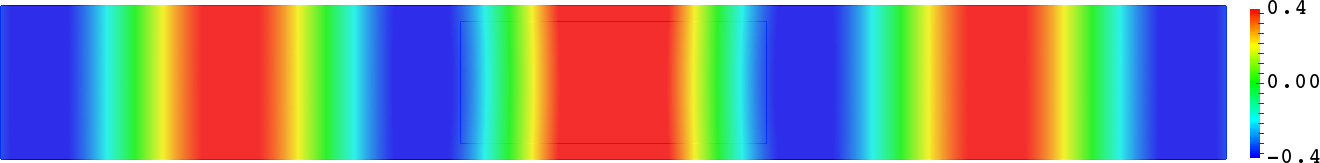}\quad\includegraphics[width=8cm]{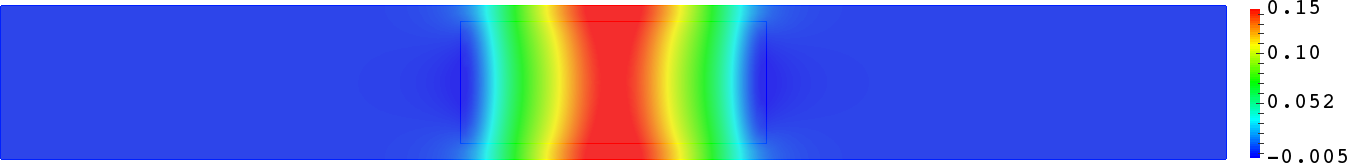}
\caption{Real parts of the total field $u^+$ (left) and of the scattered field $u^+-w^+$ (right) for the last $\rho$ of Figure \ref{rhoSimu1T}.\label{rhoSimu1FieldsT}}
\end{figure}

\noindent\textit{iii)} In Figures \ref{SimuRel1} and \ref{SimuRel2}, we impose relative invisibility, \textit{i.e.} we exhibit two different indices $\rho$ for which the corresponding scattering matrices are the same.  We work with $k=0.8\pi\in(0;\pi)$, that is again in monomode regime. In this case, the scattering matrix is of size $2\times2$ and the setting is the one of Theorem \ref{RelativeInv1D}. For the initial $\rho$ of Figure \ref{SimuRel1} (see the top left picture), we find  $T\approx 0.82+0.52i$. Since $\Re e\,T\ne0$, we work with the functional $F$ such that $F(\rho)=( \Re e\,R^+, \Im m\,R^+,\Im m\,T)$. In the last line of Figure \ref{SimuRel1}, as expected, we observe that the scattered field by the two indices have the same behaviour at infinity. For the Figure \ref{SimuRel2}, the initial $\rho$ is equal to zero (see the top left picture). In other words, there is no penetrable obstacle. But there is a defect in the wall of the waveguide and this defect has been chosen so that the transmission coefficient is null (see \cite{ChNPSu} for the explanation). This can be seen on the first picture of the second line which represents the total field for an incident wave coming from the left. We indeed note that it is exponentially decaying as $x\to+\infty$. The top right picture represents another setting, with a penetrable obstacle which has been designed so that the scattering matrix remains the same. It has been obtained working with the specific functional $F$ defined in (\ref{UniversalFunctionalTNull}). We emphasize again that the choice of the functional is crucial when imposing relative invisibility.

\begin{figure}[!ht]
\centering
\includegraphics[width=8cm]{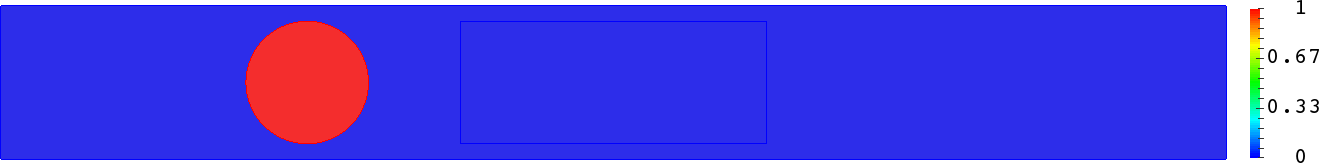}\quad\includegraphics[width=8cm]{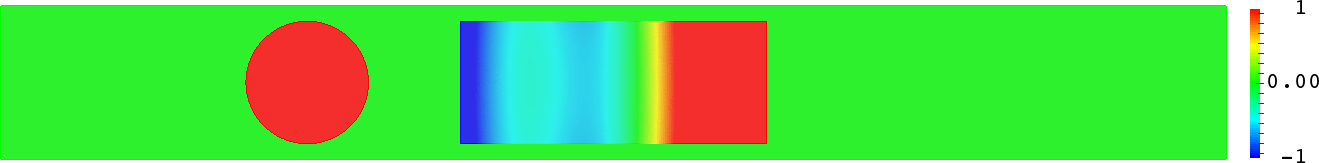}\\[4pt]
\includegraphics[width=8cm]{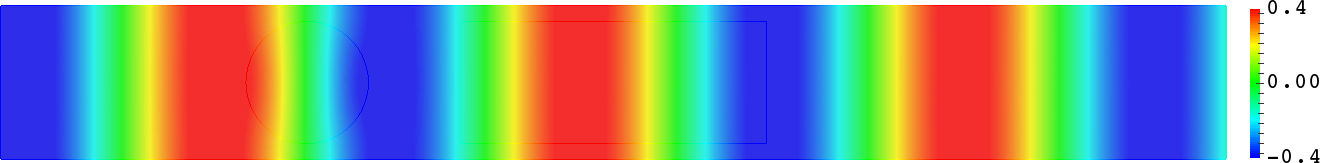}\quad
\includegraphics[width=8cm]{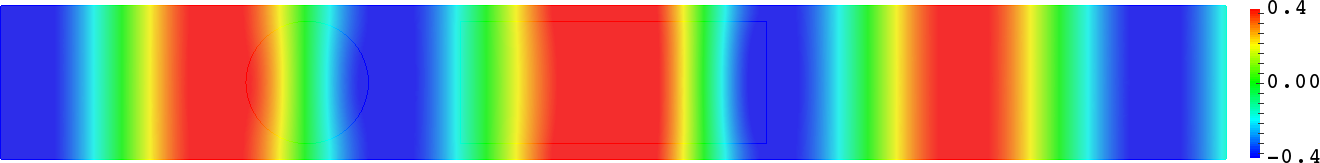}\\[4pt]
\hspace{0.05cm}\includegraphics[width=8.1cm]{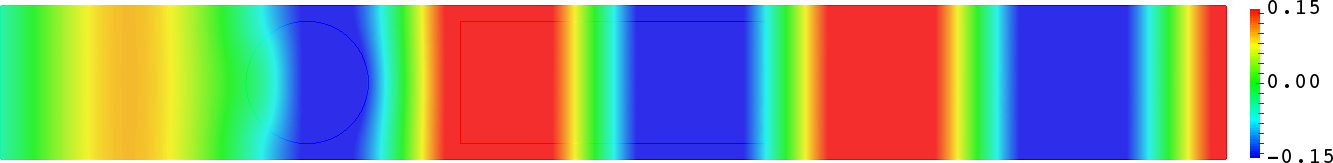}\,\,
\includegraphics[width=8.1cm]{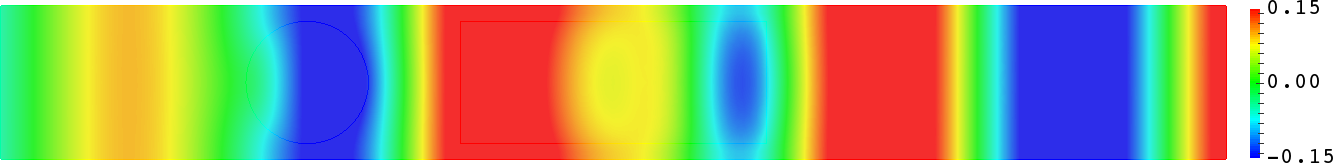}
\caption{Two different indices $\rho$ with the same scattering matrices. The first line represents the two indices. In the second (resp. third) line, we display the real part of the total field $u^+$ (resp. scattered field $u^+-w^+$) for each of the two indices. \label{SimuRel1}}
\end{figure}

\begin{figure}[!ht]
\centering
\includegraphics[width=8cm]{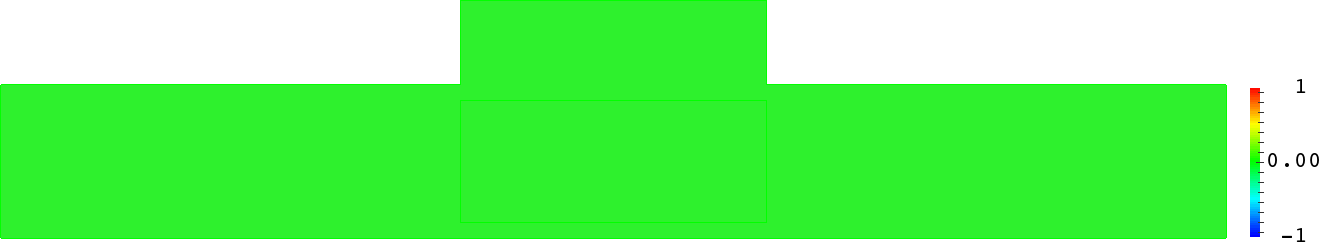}\quad\includegraphics[width=8cm]{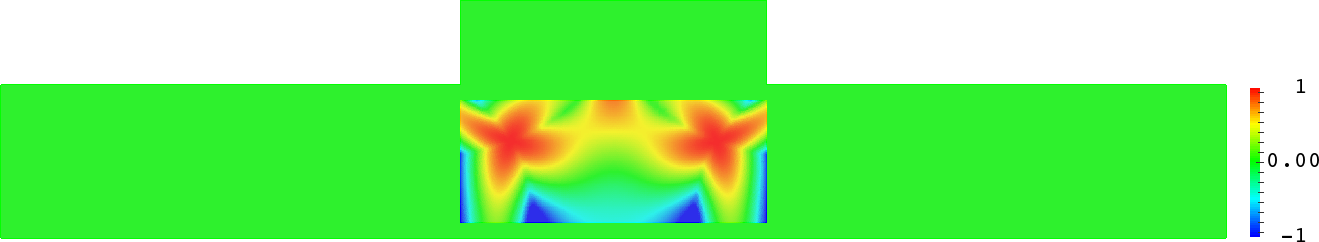}\\[4pt]
\includegraphics[width=8cm]{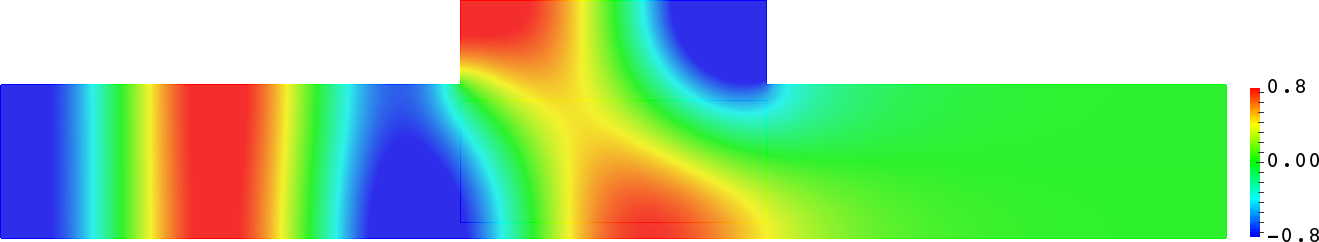}\quad
\includegraphics[width=8cm]{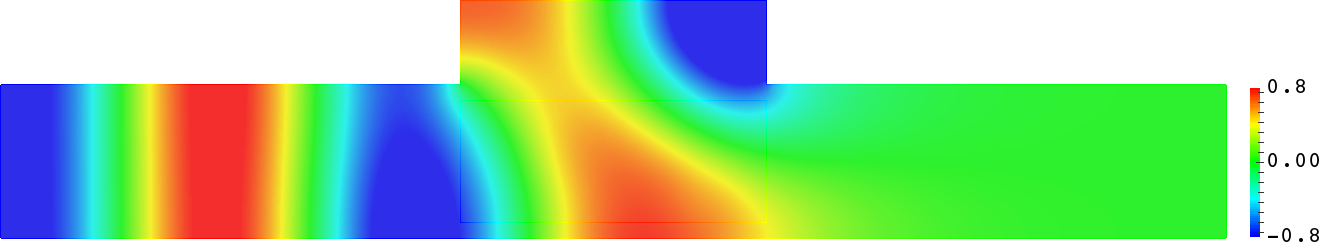}\\[4pt]
\hspace{0.05cm}\includegraphics[width=8.1cm]{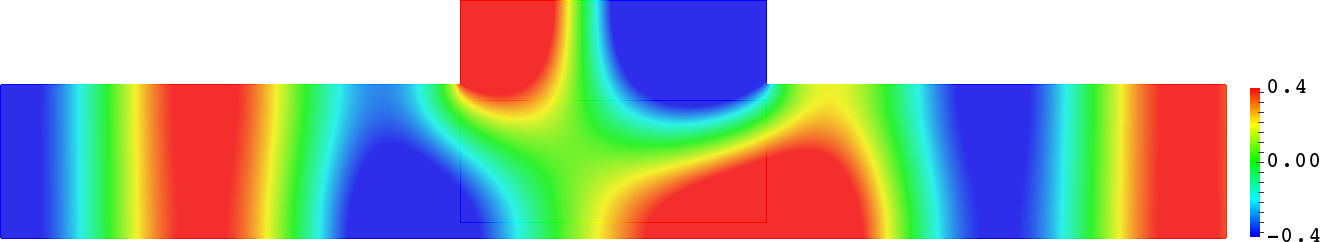}\,\,
\includegraphics[width=8.1cm]{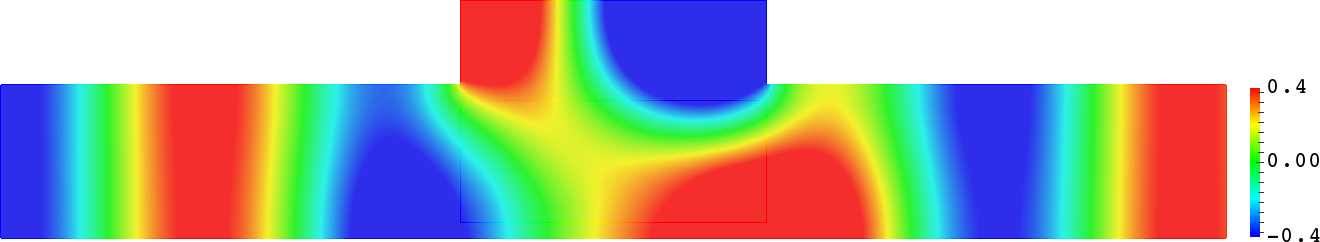}
\caption{Two different indices $\rho$ with the same scattering matrices. The first line represents the two indices. In the second (resp. third) line, we display the real part of the total field $u^+$ (resp. scattered field $u^+-w^+$) for each of the two indices. \label{SimuRel2}}
\end{figure}

\noindent\textit{iv)} In Figures \ref{SuiterhoSimu2}--\ref{ChampSimu2}, we impose invisibility in reflection but this time for $k=7\in (2\pi;3\pi)$. In this case, three modes can propagate in the waveguide and we have to cancel 6 complex terms (because $\mathbb{S}$ is symmetric). We work with the $F$ defined in (\ref{NoReflection}). We emphasize that for this $F$, we do not have a proof of ontoness of the differential. However we can still implement the method and numerically we have not noticed particular obstruction. In Figure \ref{SuiterhoSimu2}, we display the sequence of non reflecting obstacles.  
Though constraints are quite numerous, reiterating nine times the fixed point algorithm, we observe that we can get non reflecting $\rho$ with a relatively high contrast (see the last picture of Figure \ref{SuiterhoSimu2}). In Figure \ref{ChampSimu2}, we display the real part of the total and scattered fields of the three modes for the last $\rho$ of Figure \ref{SuiterhoSimu2}. As expected, we observe that the scattered fields are exponentially decaying for $x\to-\infty$.

\begin{figure}[!ht]
\centering
\includegraphics[width=5.2cm]{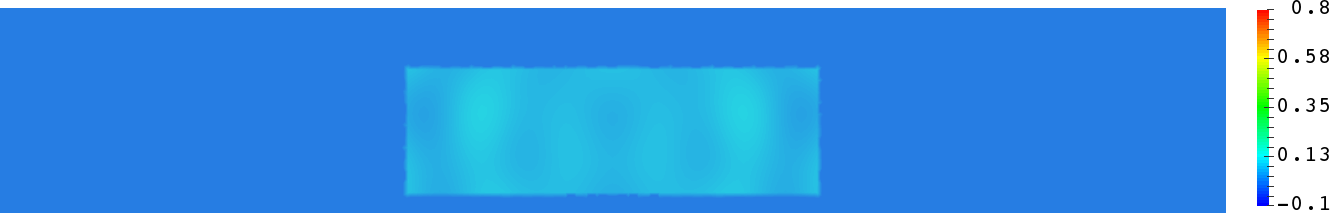}\quad\includegraphics[width=5.2cm]{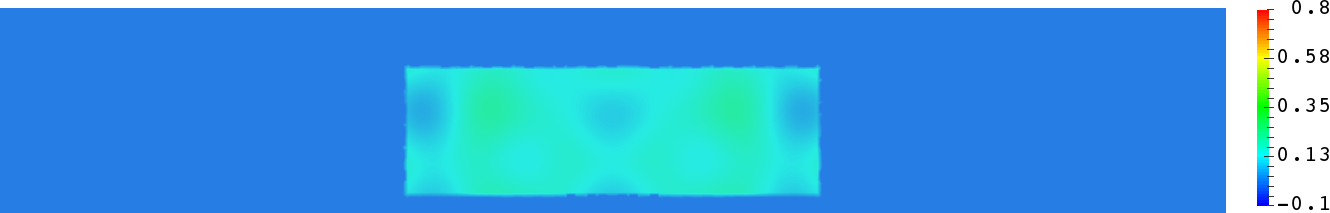}\quad\includegraphics[width=5.2cm]{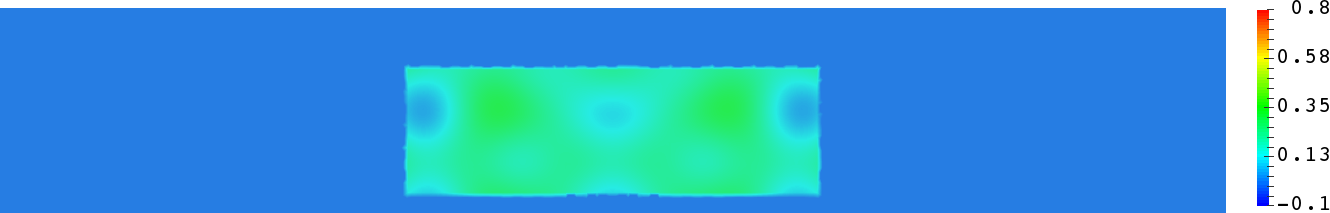}\\
\includegraphics[width=5.2cm]{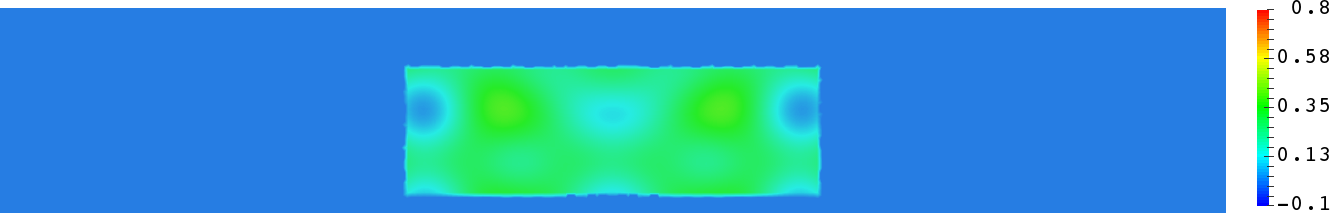}\quad\includegraphics[width=5.2cm]{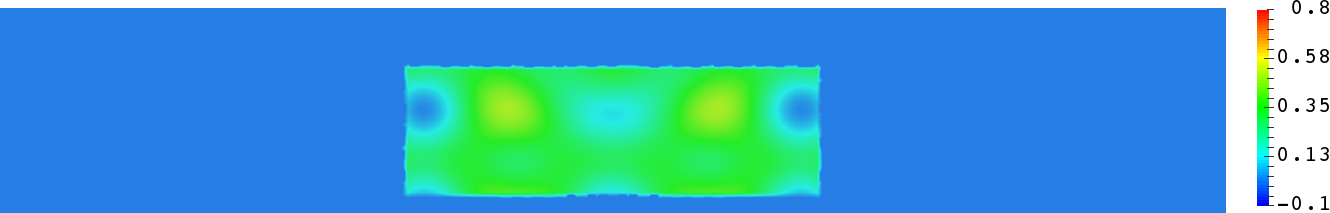}\quad\includegraphics[width=5.2cm]{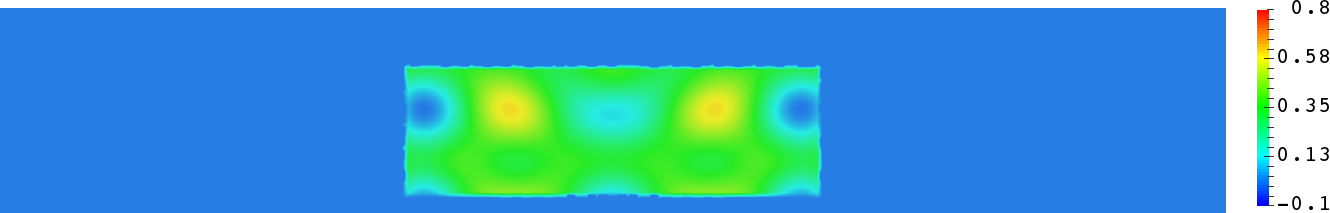}\\
\includegraphics[width=5.2cm]{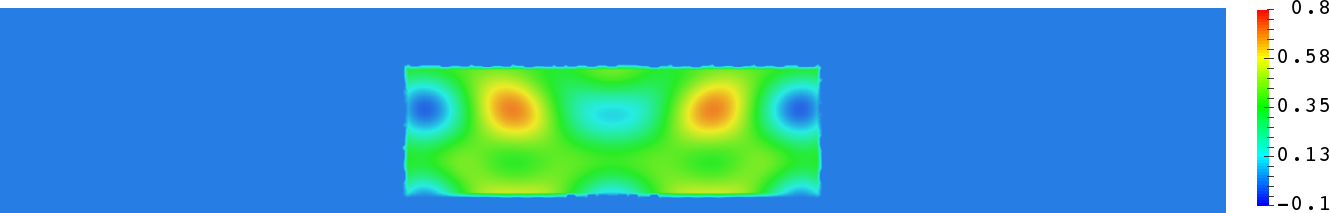}\quad\includegraphics[width=5.2cm]{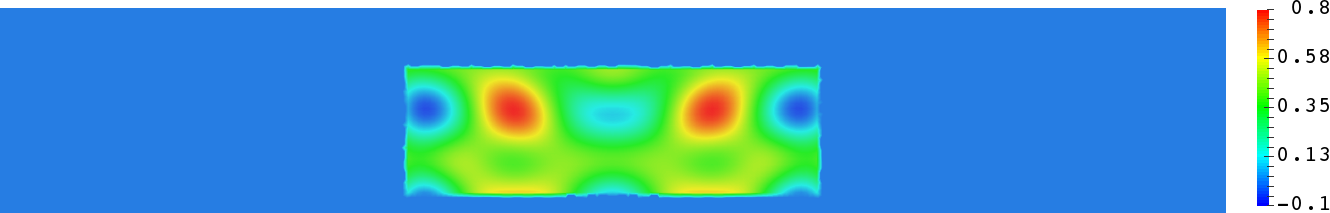}\quad\includegraphics[width=5.2cm]{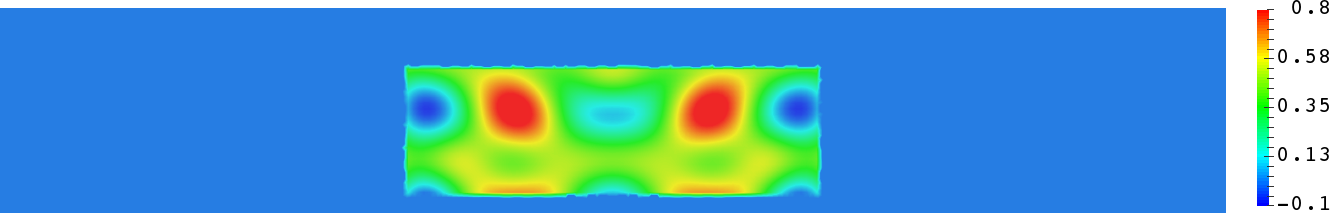}\\
\caption{Evolution of the non reflecting parameter $\rho$.  \label{SuiterhoSimu2}}
\end{figure}
\begin{figure}[!ht]
\centering
\begin{tabular}{lcc}
\raisebox{3mm}{\small\mbox{Mode 0}}&\includegraphics[width=7cm]{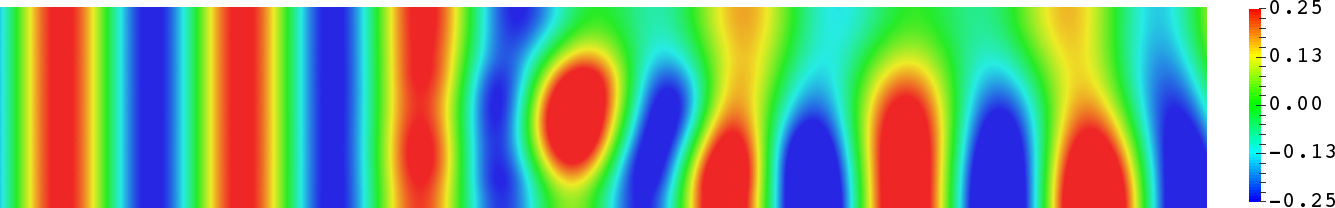}& \includegraphics[width=7cm]{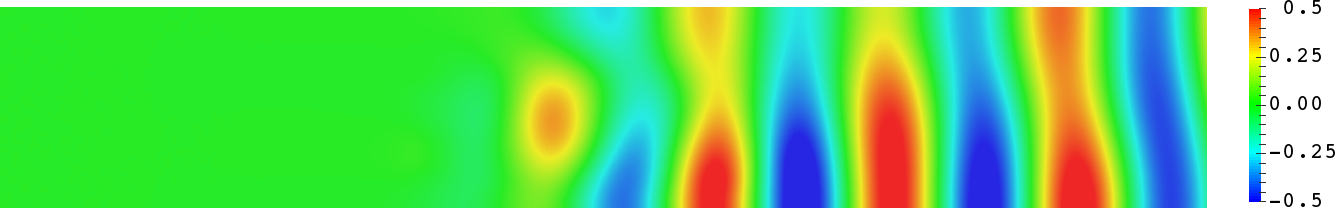}\\
\raisebox{3mm}{\small\mbox{Mode 1}}&\includegraphics[width=7cm]{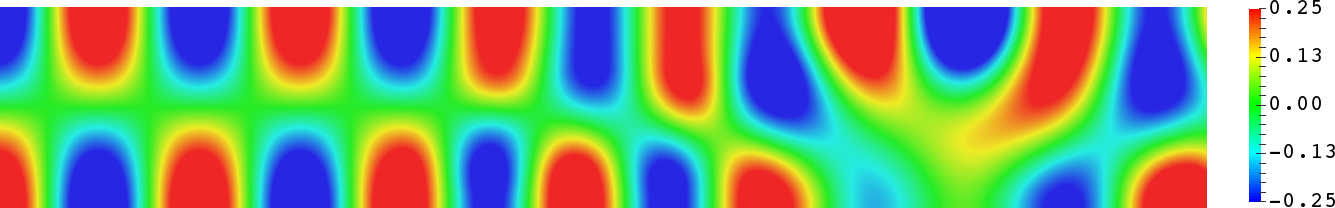}&  \includegraphics[width=7cm]{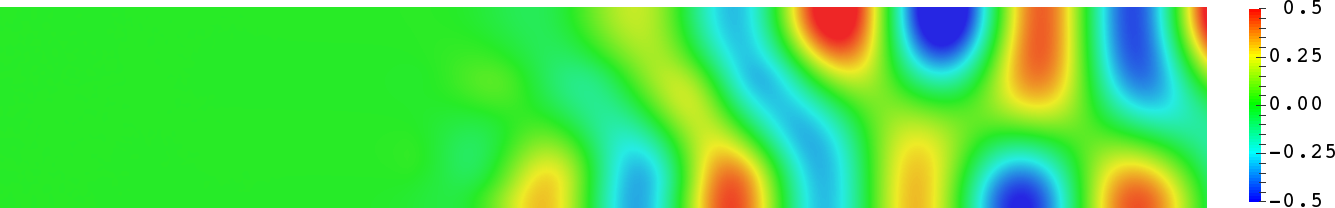}\\
\raisebox{3mm}{\small\mbox{Mode 2}} &\includegraphics[width=7cm]{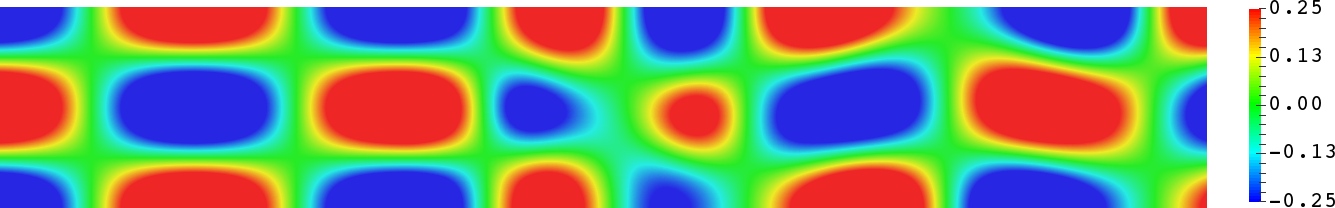}&  \includegraphics[width=7cm]{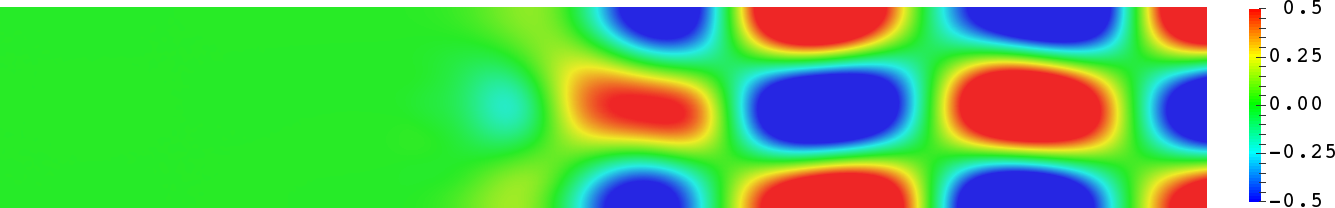}
\end{tabular}
\caption{The line $j$, $j=1,2,3$, represents the real parts of the total field $u_{j-1}^+$ (left) and of the scattered field $u_{j-1}^+-w^+_j$ (right) for the last $\rho$ of Figure \ref{SuiterhoSimu2}.\label{ChampSimu2}}
\end{figure}

\newpage
\noindent\textit{v)} In Figures \ref{rhoSimu3}--\ref{ChamprhoSimu3}, we impose invisibility in reflection for $k=4\in(\pi;2\pi)$. In this case two modes can propagate. Again we work with the $F$ defined in (\ref{NoReflection}) and we do not have a proof of ontoness of the differential. This times, the support of the non reflecting $\rho$ is the union of a rectangle and an ellipse. The choice of the support of the obstacle is not important and does not affect the method.

\begin{figure}[!ht]
\centering
\includegraphics[width=8cm]{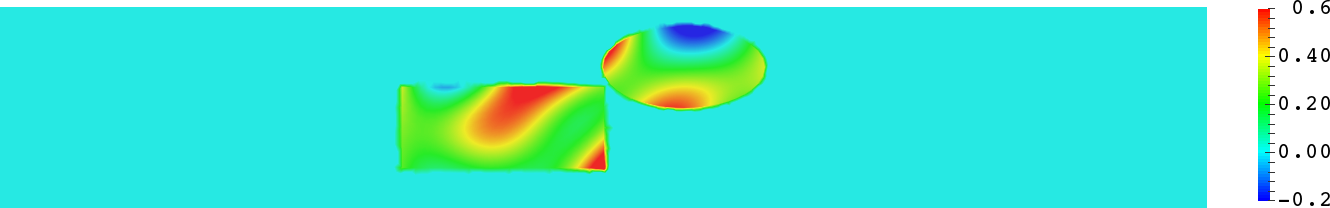}
\caption{Non reflecting parameter $\rho$. \label{rhoSimu3}}
\end{figure}
\begin{figure}[!ht]
\centering
\begin{tabular}{lcc}
\raisebox{3mm}{\small\mbox{Mode 0}}&\includegraphics[width=7cm]{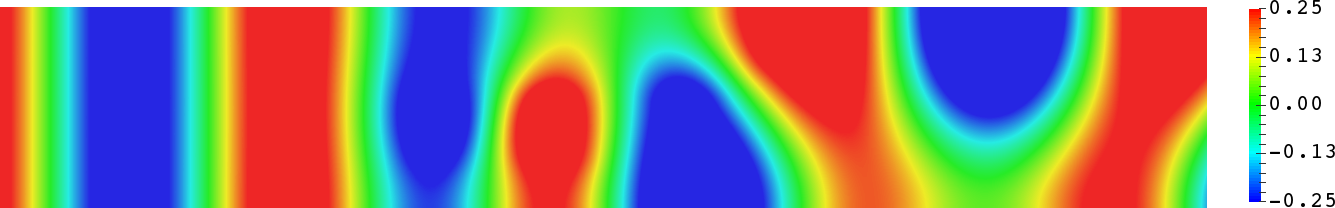}& \includegraphics[width=7cm]{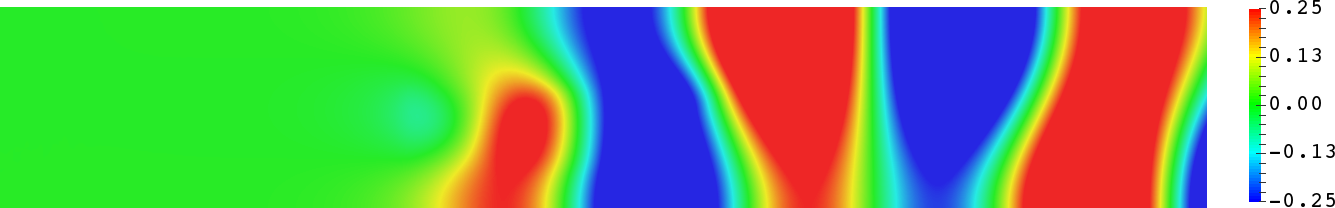}\\
\raisebox{3mm}{\small\mbox{Mode 1}}&\includegraphics[width=7cm]{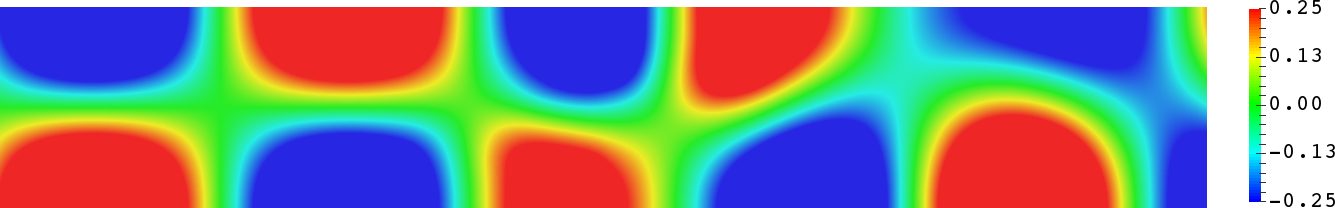}& \includegraphics[width=7cm]{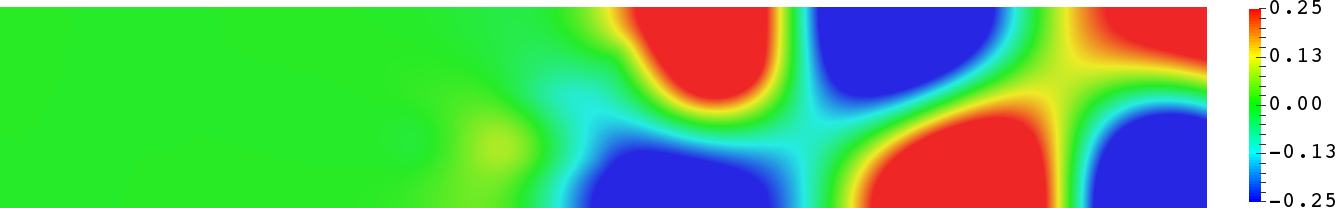}
\end{tabular}
\caption{The line $j$, $j=1,2$, represents the real parts of the total field $u_{j-1}^+$ (left) and of the scattered field $u_{j-1}^+-w^+_j$ (right) for the $\rho$ of Figure \ref{rhoSimu3}. \label{ChamprhoSimu3}}
\end{figure}

\section{Additional constraints}\label{SectionConstraints}
\subsection{General procedure}
The invisible or relatively invisible obstacles we constructed in the previous section can have some quite varying $\rho$ which can be hard to produce in practice. On the other hand, we saw with (\ref{BasisFunctionsGene}) that we have some freedom to construct the invisible perturbations. In this section, we explain how to design simpler invisible or relatively invisible $\rho$. Assume that we have a certain partition of the obstacle, \textit{i.e.} assume that we have
\[
\overline{\mathscr{O}}=\bigcup_{s=1}^{S}\overline{\mathscr{O}_{s}}
\]
where the $\mathscr{O}_{s}$ are non empty open sets such that $\mathscr{O}_{s}\cap\mathscr{O}_{s'}=\emptyset$ when $s\ne s'$. We will look for $\rho$ such that $F(\rho)=F(\rho_0)$ which are piecewise constant in the $\mathscr{O}_{s}$. Define the indicator function $\psi_s$ such that 
\[
\psi_s=\begin{array}{|ll}
1 & \mbox{ in }\mathscr{O}_{s}\\
0 & \mbox{ in }\mathscr{O}_{s'}\mbox{ for }s'\ne s
\end{array}
\] 
and set $\mX:=\mrm{span}(\psi_1,\dots,\psi_S)$. For $i=1,\dots,d$, denote by $\hat{f}_i$ the projection of the $f_i$ introduced in (\ref{diffFGene}) on the space $\mX$ for the inner product of $\mL^2(\mathscr{O})$. The $\hat{f}_i$ are such that
\[
\int_{\mathscr{O}}\hat{f}_i \psi_s\,dz=\int_{\mathscr{O}}f_i \psi_s\,dz,\qquad \forall s=1,\dots,S.
\]
Define the new Gram matrix
\[
\hat{\G}:=\left(\int_{\mathscr{O}} \hat{f}_i\hat{f}_j\,dz \right)_{1 \le i,j\le d}.
\]
Denote $\hat{\mathbb{H}}=(\hat{\mathbb{H}}_{ij})_{1 \le i,j\le d}$ the inverse of $\hat{\G}$ assuming that it exists. Note that $\hat{\G}$ is invertible if and only if $\{\hat{f}_1,\dots,\hat{f}_d\}$ is a family of linearly independent functions which is not guaranteed even when $dF(\rho_0):\mL^{\infty}(\mathscr{O})\to\R^d$ is onto. Observe that a necessary condition so that this holds true is that $S\ge d$: the number of elements in the partition of $\mathscr{O}$ must be larger than the number of constraints to satisfy. Finally, we set
\[
\hat{\mu}_j=\sum_{i=1}^d\hat{\mathbb{H}}_{ji}\hat{f}_i.
\]
Then from a $\hat{\mu}_0^{\#}\in\mX$ such that $\hat{\mu}_0^{\#}\notin \mrm{span}(\hat{\mu}_1,\dots,\hat{\mu}_d)$, we define $\hat{\mu}_0$  by 
\[
\hat{\mu}_0:=\hat{\mu}_0^{\#}-\sum_{j=1}^ddF_j(\rho_0)(\hat{\mu}_0^{\#})\,\hat{\mu}_j.
\]
With such definitions, the functions $\hat{\mu}_0$, $\hat{\mu}_1,\dots,\hat{\mu}_d$ satisfy (\ref{BasisFunctionsGene}). Then the rest of the algorithm is as described at the end of \S\ref{paragraphGenProc}.

\subsection{Results}

Let us present two series of experiments to show the workability of the algorithm. In Figures \ref{rhoSimu6}--\ref{ChamprhoSimu6}, we construct piecewise constant perfectly invisible parameters $\rho$. We work with $k=0.8\pi\in(0;\pi)$. Since we cancel both $R^+$ and $\Im m\,T$, we need to work with at least three inclusions. In Figures \ref{rhoSimu5}--\ref{ChamprhoSimu5}, we construct piecewise constant non reflecting $\rho$ with $k=7\in(2\pi;3\pi)$. In this case, three modes can propagate in the waveguide. Since we have to cancel 6 complex coefficients, that is 12 real coefficients, we need to have at least 12 parameters to tune. This is why we work with 30 penetrable circular inclusions. We emphasize that due to the additional constraints imposed to the index $\rho$, we have no proof of ontoness of the differential of the functionals. As a consequence, we can not justify the result of existence of non reflecting $\rho$ satisfying the constraints. However for the cases we have considered numerically, the algorithm has worked very reasonably and we have not observed any particular obstruction.

\begin{figure}[!ht]
\centering
\includegraphics[width=8cm]{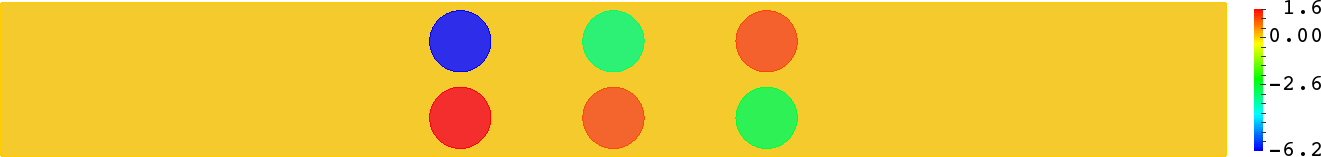}
\caption{Piecewise constant perfectly invisible parameter $\rho$. \label{rhoSimu6}}
\end{figure}
\begin{figure}[!ht]
\centering
\includegraphics[width=7cm]{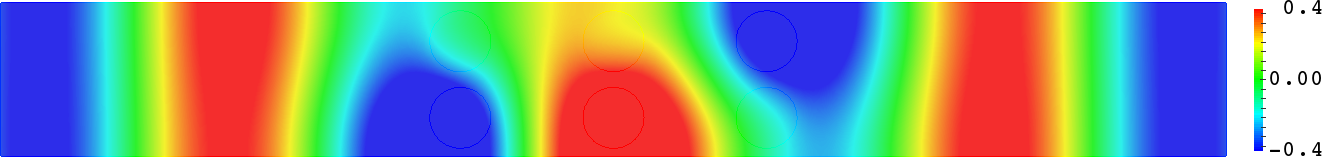}\qquad\includegraphics[width=7cm]{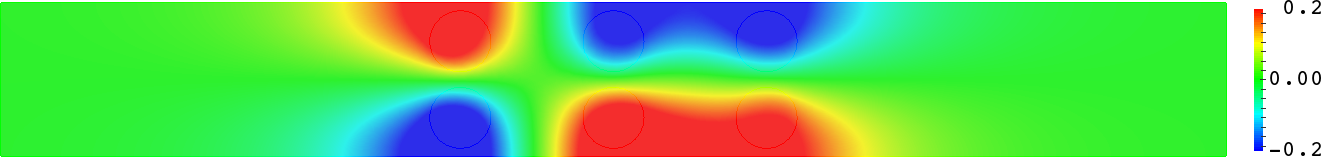}
\caption{Real parts of the total field $u^+$ (left) and of the scattered field $u^+-w^+$ (right) for the $\rho$ of Figure \ref{rhoSimu6}.\label{ChamprhoSimu6}}
\end{figure}

\begin{figure}[!ht]
\centering
\includegraphics[width=8cm]{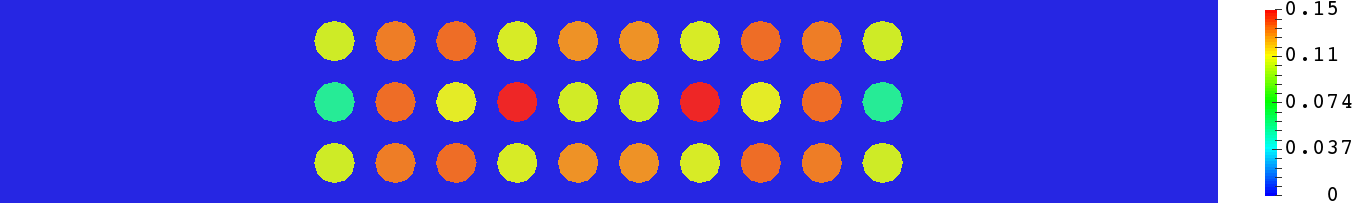}
\caption{Piecewise constant non reflecting parameter $\rho$. \label{rhoSimu5}}
\end{figure}
\begin{figure}[!ht]
\centering
\begin{tabular}{lcc}
\raisebox{3mm}{\small\mbox{Mode 0}}&\includegraphics[width=7cm]{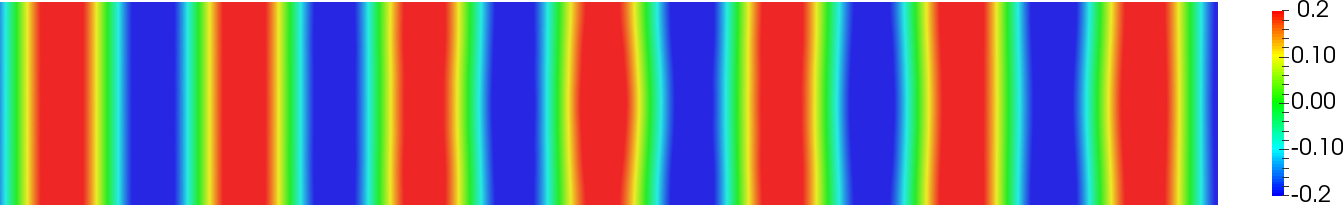}& \includegraphics[width=7cm]{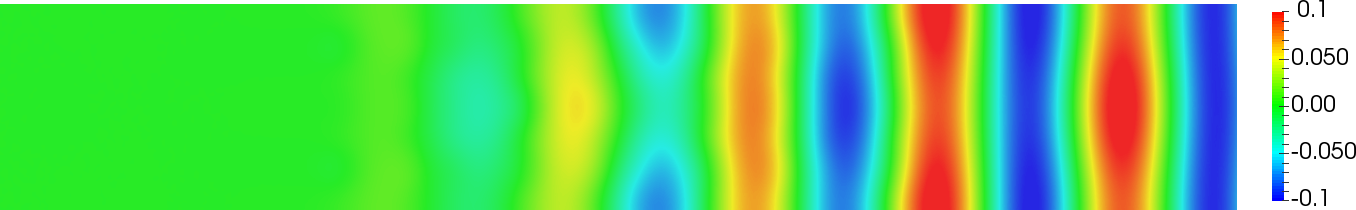}\\
\raisebox{3mm}{\small\mbox{Mode 1}}&\includegraphics[width=7cm]{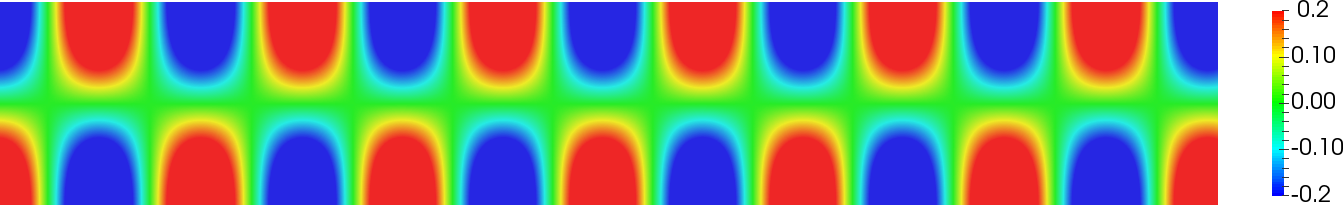}& \includegraphics[width=7cm]{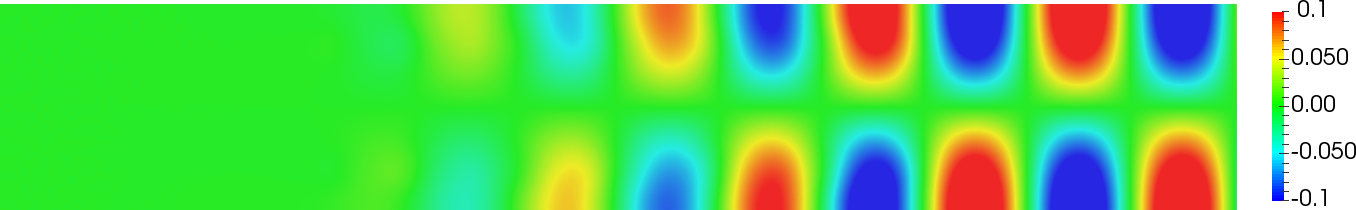}\\
\raisebox{3mm}{\small\mbox{Mode 2}}&\includegraphics[width=7cm]{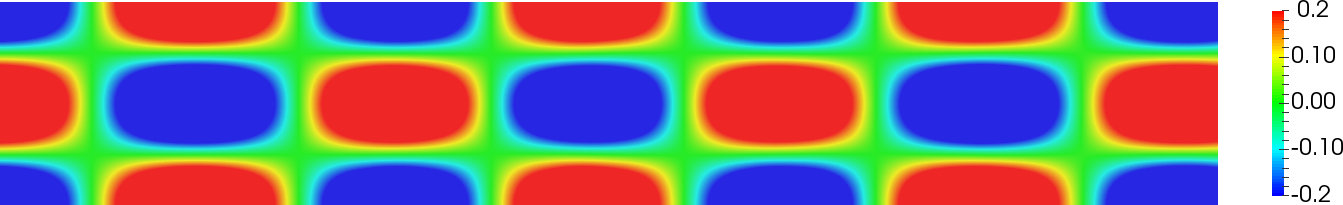}& \includegraphics[width=7cm]{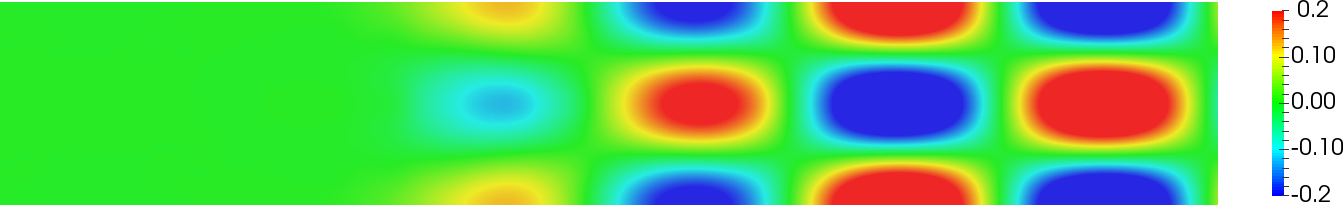}
\end{tabular}
\caption{The line $j$, $j=1,2,3$, represents the real parts of the total field $u_{j-1}^+$ (left) and of the scattered field $u_{j-1}^+-w^+_j$ (right) for the $\rho$ of Figure \ref{rhoSimu5}. \label{ChamprhoSimu5}}
\end{figure}

\newpage
\section{Concluding remarks}\label{Conclusion}

In this work, we have presented a continuation technique to construct non reflecting, invisible or relatively invisible penetrable obstacles in acoustic waveguides. We have provided a complete proof of the method in monomode regime. In multimode regime, that is when the wavenumber is such that several modes can propagate, we have given all the ingredients to implement the method and numerically it gives satisfying results. However in this case, there is still some theoretical work to establish the ontoness of the differentials of the functionals which are involved in the construction. These results, as well as the results of ontoness when one imposes additional constraints on the index (see Section \ref{SectionConstraints}), seem hard to obtain. On the other hand, it would be interesting to explore which kind of constraints it is relevant to impose to the index for applications. For example, how to proceed to prevent the invisible index to become negative? We have focused our attention on the construction of invisible perturbations of the index material. We could have considered in a similar way the question of building invisible perturbations of the geometry (see \cite{BoNa13}). Note that this problem is slightly different (for example, as explained in \cite{na648}, it is harder to impose perfect invisibility) and questions concerning the choices of the functionals as well as proofs of ontoness should be studied carefully. We have worked with equations of acoustic in 2D. The analysis is completely the same in higher dimensions and can be simply adapted to deal with problems of quantum waveguides (Dirichlet boundary condition) or water-waves. Besides, we have imposed invisibility at a given wavenumber. We could have imposed similarly invisibility for several wavenumbers. However we emphasize that the set of measurements should remain discrete. The approach does not allow one to impose invisibility for a continuum of wavenumbers (which may be impossible, see the related works \cite{SoGK07,Norr15,MoAl16,CaMi17,Norr18}).

\section*{Acknowledgements}
The research of Antoine Bera was supported by the DGA, Direction G\'en\'erale de l'Armement.

\bibliography{Bibli}
\bibliographystyle{plain}

\end{document}